\documentclass[a4paper,11pt,oneside]{article}
\usepackage{graphicx}
\usepackage{amscd}
\usepackage{amsmath}
\usepackage{caption}
\usepackage{amsfonts}
\usepackage{amssymb}
\usepackage{amsthm}
\usepackage{mathrsfs}
\usepackage{multicol}
\usepackage{color}
\usepackage[english]{babel}
\usepackage[T1]{fontenc}
\usepackage{textcomp}
\usepackage[utf8]{inputenc}
\usepackage{enumitem}
\usepackage{indentfirst}

\newcommand{\N}{\mathbb N}

\newcommand{\RR}{{{\rm I} \kern -.15em {\rm R} }}

\begin{document}
	\theoremstyle{plain} \newtheorem{thm}{Theorem}[section] \newtheorem{cor}[thm]{Corollary} \newtheorem{lem}[thm]{Lemma} \newtheorem{prop}[thm]{Proposition} \theoremstyle{definition} \newtheorem{defn}{Definition}[section] 
	
	\newtheorem{oss}[thm]{Remark}
	\newtheorem{ex}{Example}[section]
	\newtheorem{lemma}{Lemma}[section]
	\title{Hegselmann-Krause and Cucker-Smale type models with attractive-repulsive interaction}
	\author{Elisa Continelli\hspace{0.2cm}\&\hspace{0.2cm}Cristina Pignotti \\Dipartimento di Ingegneria e Scienze dell'Informazione e Matematica\\
		Universit\`{a} degli Studi di L'Aquila\\
		Via Vetoio, Loc. Coppito, 67100 L'Aquila Italy}

	\maketitle

	\begin{abstract}
		In this paper, we analyze a Hegselmann-Krause opinion formation model and a Cucker-Smale flocking model with attractive-repulsive interaction. To be precise, we investigate the situation in which the individuals involved in an opinion formation or a flocking process attract each other in certain time intervals and repeal each other in other ones. Under quite general assumptions, we prove the convergence to consensus for the Hegselmann-Krause model and the exhibition of asymptotic flocking for the Cucker-Smale model in presence of positive-negative interaction. With some additional conditions, we are able to improve the convergence to consensus for the solutions of the Hegselmann-Krause model, namely we establish an exponential convergence to consensus result. 
	\end{abstract}
\section{Introduction}
	Multiagent systems have been widely studied in these last years, due to their application to several scientific fields, such as biology \cite{Cama, CS1}, economics \cite{Marsan, Jack}, robotics \cite{Bullo, Jad}, control theory \cite{Aydogdu, Borzi, WCB, PRT, Piccoli, Almeida2}, social sciences \cite{Bellomo, Campi, CF, Ajmone, Lac}. Among them, there are the Hegselmann-Krause opinion formation model \cite{HK} and its second order version, the Cucker-Smale flocking model, introduced in \cite{CS1} to describe flocking phenomena. For the solutions of such systems, the asymptotic behaviour is usually investigated, namely, the convergence to consensus or the exhibition of asymptotic flocking are analyzed for the Hegselmann-Krause model and the Cucker-Smale model, respectively. 
	
	For the two aforementioned multiagent systems, different scenarios have been considered. Among others, there is the situation in which the particles of the system interrupt their interaction at certain times. In \cite{Bonnet}, the convergence to consensus and the exhibition of asymptotic flocking for a class of Cucker-Smale systems with communication failures have been proved under suitable assumptions. With a different approach, the exponential convergence to consensus for the solutions of Hegselmann-Krause type models  involving a weight function that can degenerate at some times has been achieved in \cite{onoff}, also in presence of time delays.

	 The Hegselmann-Krause model and the Cucker-Smale model in presence of time delays, but in the not degenerate case,  have been also analyzed by several authors (see e.g. \cite{Cont, ContPign, Lu, LW, CH, CL, PT, HM, DH,  CPP, H3, Cartabia})
	
	A situation that to our knowledge has been not examined yet is the one in which the agents have positive-negative interaction, namely the system's particles attract each other in certain time intervals and repeal each other in other ones. 
	
	The aim of this paper is indeed to find conditions for which the convergence to consensus and the exhibition of asymptotic flocking are guaranteed for the solutions of the Hegselmann-Krause model and of the Cucker-Smale model, respectively, with attractive-repulsive interaction. In this case, in order to get the asymptotic consensus or the asymptotic flocking, one has to compensate the behaviour of the solutions to the considered model in the {\it bad} intervals, i.e. the intervals in which the agents repeal each other, with the {\it good} behaviour of the solutions of the system in the intervals in which the influence among the agents is positive. Under suitable assumptions, we are able to prove the convergence to consensus for the Hegselmann-Krause model with attractive-repulsive interaction and the exhibition of asymptotic flocking for the Cucker-Smale model with attractive-repulsive interaction. Of course, this will require some restrictions on the length of the {\it bad} time intervals.
Under some additional conditions, it can be proved that the solutions of the Hegselmann-Krause model in presence of positive-negative interaction converge exponentially to consensus.
	
	The paper is organized as follows. In Section 2, we deal with a Hegselmann-Krause model with attractive-repulsive interaction in presence of a very general type of influence function. We give the definition of diameter for the solutions to the model and the rigorous definition of convergence to consensus. Also, we state our main result, that ensures the asymptotic consensus for our Hegselmann-Krause type model under quite general assumptions. 

In Section 3, we present some preliminary results that are needed for the proof of our convergence to consensus result while,  in Section 4, we prove our main result. Also, under less general assumptions, we provide an exponential convergence to consensus result. 

In Section 5 we formulate the convergence to consensus for a Hegselmann-Krause model with attractive-repulsive interaction in the particular case in which the influence function is no longer a general function of the opinions of the agents but a function of the distance of the agents' opinions than can eventually tend to zero as the distance of the opinions becomes too large. 

In Section 6, we consider a Cucker-Smale model in presence of attractive-repulsive interaction. In this case, the influence function cannot be a general function of the 
agents' opinion but must rather be a function of the distance of the opinions of the system's particles. However, the monotonicity of the influence function, that is assumed in general when dealing with Cucker-Smale type models, is removed (see also \cite{Cont}). We give the definitions of the diameter in position, the diameter in velocity, and the exhibition of asymptotic flocking for the solutions of the Cucker-Smale model. Also, we state our main result, that ensures that every solution to the Cucker-Smale model with attractive-repulsive interaction reaches the asymptotic flocking under suitable assumptions.

Finally, in Section 7 we prove our flocking theorem through the use of some preliminary results.
	\section{The Hegselmann-Krause model with attractive-repulsive interaction and general  influence function}
	Let us denote with $\mathbb{N}_0:=	\mathbb{N}\cup \{0\}$. Consider a finite set of $N\in\N$ particles, with $N\geq 2 $. Let $x_{i}(t)\in \RR^d$ be the opinion of the $i$-th individual at time $t$. We shall denote with $\lvert\cdot \rvert$ and $\langle\cdot,\cdot\rangle$ the usual norm and scalar product on $\RR^{d}$, respectively. The interactions between the elements of the system are described by the following  Hegselmann-Krause type model 
	\begin{equation}\label{onoff}
		\frac{d}{dt}x_{i}(t)=\underset{j:j\neq i}{\sum}\alpha(t) a_{ij}(t)(x_{j}(t)-x_{i}(t)),\quad t>0,	\,\, \forall i=1,\dots,N.
	\end{equation}
	with initial conditions
	\begin{equation}\label{incond}
		x_{i}(0)=x^{0}_{i}\in \mathbb{R}^{d},\quad \forall i=1,\dots,N.
	\end{equation}
Here, the communication rates $a_{ij}$ are of the form
	\begin{equation}\label{weight}
		a_{ij}(t):=\frac{1}{N-1}\psi( x_{i}(t), x_{j}(t)), \quad\forall t>0,\, \forall i,j=1,\dots,N,
	\end{equation}
	where $\psi:\RR^d\times\RR^d\rightarrow \RR$ is a given influence function. We assume that the influence function $\psi$ is positive, continuous, and bounded and we denote 
	\begin{equation}\label{K}
		K:=\lVert \psi\rVert_{\infty}.
	\end{equation} 
Moreover, the weight function $\alpha:[0,+\infty)\rightarrow\{-1,1\}$ is defined as follows
	\begin{equation}\label{alpha1}
		\alpha(0)=1,\quad\alpha(t)=\begin{cases}
			1,\quad &t\in \left(t_{2n},t_{2n+1}\right), \,\ \, \ n\in \mathbb{N}_{0},\\-1,\quad &t\in \left[t_{2n+1},t_{2n+2}\right], \,\, n\in \mathbb{N}_{0},
		\end{cases}
	\end{equation}
	where $\{t_{n}\}_{n}$ is a sequence of nonnegative numbers such that $t_{0}=0$, $t_{n}\to \infty$ as $n\to \infty$ and 
	\begin{equation}\label{tn}
		t_{2n+2}-t_{2n+1}<\frac{\ln 2}{K},\quad \forall n\in \mathbb{N}_{0}.
	\end{equation}
	Also, we set
	\begin{equation}\label{M0}
		M_{n}^{0}:=\max_{i=1,\dots,N}\lvert x_{i}(t_{n})\rvert.
	\end{equation}
	We define the diameter $d(\cdot)$ of the system as follows:
	$$d(t):=\max_{i,j=1,\dots,N}\lvert x_{i}(t)-x_{j}(t)\rvert,\quad \forall t\geq 0.$$
	
	\begin{defn}\label{consensus}
		We say that a solution $\{x_{i}\}_{i=1,\dots,N}$ to system \eqref{onoff} converges to \textit{consensus} if $$\lim_{t\to+\infty}d(t)=0.$$
	\end{defn}
	We will prove the following convergence to consensus result.
	\begin{thm}\label{cons}
		Let $\psi:\RR^d\times \RR^d\rightarrow\RR$ be a positive, bounded, continuous function. Assume that the sequence $\{t_{n}\}_{n}$ of definition \eqref{alpha1} satisfies \eqref{tn}. Assume also that the following conditions hold:
		\begin{equation}\label{sum1}
			\sum_{p=0}^{\infty}\ln\left(\frac{e^{K(t_{2p+2}-t_{2p+1})}}{2-e^{K(t_{2p+2}-t_{2p+1})}}\right)<+\infty,
		\end{equation}
	\begin{equation}\label{sum2}
		\sum_{p=0}^{\infty}\ln\left(\max\left\{1-e^{-K(t_{2p+1}-t_{2p})},1-\frac{\psi_{0}}{K}(1-e^{-K(t_{2p+1}-t_{2p})})\right\}\right)=-\infty,
	\end{equation} 
where
\begin{equation}\label{psi0}
	 \psi_{0}:=\min_{\lvert y\rvert,\lvert z\rvert\leq M^{0}} \psi(y,z),
\end{equation}
being 
\begin{equation}\label{m0}
	M^{0}=e^{K\sum_{p=0}^{\infty}(t_{2p+2}-t_{2p+1})}M^{0}_{0}.
\end{equation}Then, every solution $\{x_{i}\}_{i=1,\dots,N}$ to \eqref{onoff} with the initial conditions \eqref{incond} converges to consensus.
		\end{thm}
	\begin{oss}
		Let us note that condition \eqref{sum1} implies that \begin{equation}\label{bad}
			\sum_{p=0}^{+\infty}(t_{2p+2}-t_{2p+1})<+\infty,
		\end{equation}
		so that the quantity $M^{0}$ in \eqref{m0} is finite and it makes sense to consider the minimum given by \eqref{psi0}. Indeed, being $e^{K(t_{2p+2}-t_{2p+1})}>1$, it turns out that $$e^{K(t_{2p+2}-t_{2p+1})}\leq \frac{e^{K(t_{2p+2}-t_{2p+1})}}{2-e^{K(t_{2p+2}-t_{2p+1})}},$$
		from which $$K\sum_{p=0}^{+\infty}(t_{2p+2}-t_{2p+1})\leq \sum_{p=0}^{+\infty}\ln\left(\frac{e^{K(t_{2p+2}-t_{2p+1})}}{2-e^{K(t_{2p+2}-t_{2p+1})}}\right).$$
		Then, from \eqref{sum1}, the condition \eqref{bad} is satisfied. As a consequence, $t_{2p+2}-t_{2p+1}\to 0$, as $p\to \infty$. 
		
	\end{oss}
\begin{oss}\label{boundpsi}
Assume that $$t_{2n+1}-t_{2n}>\frac{1}{K}\ln\left(1+\frac{K}{\psi_{0}}\right),\quad \forall n\in \mathbb{N}_{0}.$$
Note that $\ln\left(1+\frac{K}{\psi_{0}}\right)>\ln2$, so that $$t_{2p+2}-t_{2p+1}<t_{2q+1}-t_{2q},\quad \forall p,q\in \mathbb{N}_0.$$ 
Then, in this situation the condition \eqref{sum2} can be simplified since $$\max\left\{1-e^{-K(t_{2p+1}-t_{2p})},1-\frac{\psi_{0}}{K}(1-e^{-K(t_{2p+1}-t_{2p})})\right\}=1-e^{-K(t_{2p+1}-t_{2p})}.$$
So, \eqref{sum2} becomes
	\begin{equation}\label{aut}
		\sum_{p=0}^{+\infty}\ln\left(1-e^{-K(t_{2p+1}-t_{2p})}\right)=-\infty.
	\end{equation} 
However, the above condition \eqref{aut} is automatically satisfied since we can assume that $$t_{2n+2}-t_{2n+1}\leq T,\quad \forall n \in \mathbb{N}_{0},$$
for some $T>0$, eventually splitting the intervals of positive interaction into subintervals of length at most $T$.
\end{oss}
	\section{Preliminaries}
	\setcounter{equation}{0}
	
	\noindent Let $\{x_{i}\}_{i=1,\dots,N}$ be solution to \eqref{onoff} under the initial conditions \eqref{incond}. In this section, we present some preliminary lemmas.  We assume that the hypotheses of Theorem \ref{cons} are satisfied. We first give some results that are related to the behaviour of the solution $\{x_{i}\}_{i=1,\dots,N}$ in the intervals having a positive interaction.
	\begin{lem}\label{L1}
		For each $v\in \RR^{d}$ and $n\in \mathbb{N}_0,$  we have that 
		\begin{equation}\label{scalpr}
			\min_{j=1,\dots,N}\langle x_{j}(t_{2n}),v\rangle\leq \langle x_{i}(t),v\rangle\leq \max_{j=1,\dots,N}\langle x_{j}(t_{2n}),v\rangle,
		\end{equation}for all  $t\in [t_{2n},t_{2n+1}]$ and $i=1,\dots,N$. 
	\end{lem}
	\begin{proof}
		Let $n\in \mathbb{N}_{0}$. Given a vector $v\in \RR^{d}$, we set $$M_{t_{2n}}=\max_{j=1,\dots,N}\langle x_{j}(t_{2n}),v\rangle.$$
	For all $\epsilon >0$, let us define
		$$K^{\epsilon}:=\left\{t\in [t_{2n},t_{2n+1}] :\max_{i=1,\dots,N}\langle x_{i}(s),v\rangle < M_{t_{2n}}+\epsilon,\,\forall s\in [t_{2n},t)\right\}.$$
		Denoted with $$S^{\epsilon}:=\sup K^{\epsilon},$$
		by continuity it holds that $S^{\epsilon}\in (t_{2n},t_{2n+1}]$. \\We claim that $S^{\epsilon}=t_{2n+1}$. Indeed, suppose by contradiction that $S^{\epsilon}<t_{2n+1}$. Note that by definition of $S^{\epsilon}$ it turns out that \begin{equation}\label{max}
			\max_{i=1,\dots,N}\langle x_{i}(t),v\rangle<M_{t_{2n}}+\epsilon,\quad \forall t\in (t_{2n},S^{\epsilon}),
		\end{equation}
		and \begin{equation}\label{teps}
			\lim_{t\to S^{\epsilon-}}\max_{i=1,\dots,N}\langle x_{i}(t),v\rangle=M_{t_{2n}}+\epsilon.
		\end{equation}
		For all $i=1,\dots,N$ and $t\in (t_{2n},S^{\epsilon})$, we compute
		$$\frac{d}{dt}\langle x_{i}(t),v\rangle=\frac{1}{N-1}\sum_{j:j\neq i}\alpha(t)\psi(x_{i}(t), x_{j}(t))\langle x_{j}(t)-x_{i}(t),v\rangle.$$
		Notice that, being $t\in (t_{2n},S^{\epsilon})$, then  $t\in (t_{2n}, t_{2n+1})$ and, as a consequence, it holds that $\alpha(t)=1.$ Thus, using \eqref{K} and \eqref{max}, we can write $$\frac{d}{dt}\langle x_{i}(t),v\rangle\leq \frac{1}{N-1}\sum_{j:j\neq i}\psi(x_{i}(t), x_{j}(t))(M_{t_{2n}}+\epsilon-\langle x_{i}(t),v\rangle)$$$$\leq K(M_{t_{2n}}+\epsilon-\langle x_{i}(t),v\rangle).$$
		Then, from Gronwall's inequality, we get
		$$\begin{array}{l}
			\vspace{0.2cm}\displaystyle{
				\langle x_{i}(t),v\rangle\leq e^{-K(t-t_{2n})}\langle x_{i}(t_{2n}),v\rangle+K(M_{t_{2n}}+\epsilon)\int_{t_{2n}}^{t}e^{-K(t-s)}ds}\\
			\vspace{0.3cm}\displaystyle{\hspace{1.7 cm}
				=e^{-K(t-t_{2n})}\langle x_{i}(t_{2n}),v\rangle+(M_{t_{2n}}+\epsilon)e^{-Kt}(e^{Kt}-e^{Kt_{2n}})}\\
			\vspace{0.3cm}\displaystyle{\hspace{1.7 cm}
				=e^{-K(t-t_{2n})}\langle x_{i}(t_{2n}),v\rangle+(M_{t_{2n}}+\epsilon)(1-e^{-K(t-t_{2n})})}\\
			\vspace{0.3cm}\displaystyle{\hspace{1.7 cm}
				\leq e^{-K(t-t_{2n})}M_{t_{2n}}+M_{t_{2n}}+\epsilon -M_{t_{2n}}e^{-K(t-t_{2n})}-\epsilon e^{-K(t-t_{2n})}}\\
			\vspace{0.3cm}\displaystyle{\hspace{1.7 cm}
				=M_{t_{2n}}+\epsilon-\epsilon e^{-K(t-t_{2n})}}\\
			\displaystyle{\hspace{1.7 cm}
				=M_{t_{2n}}+\epsilon-\epsilon e^{-K(S^{\epsilon}-t_{2n})},}
		\end{array}
		$$for all $t\in (t_{2n}, S^{\epsilon})$.	We have so proved that, $\forall i=1,\dots, N,$
		$$\langle x_{i}(t),v\rangle\leq M_{t_{2n}}+\epsilon-\epsilon e^{-K(S^{\epsilon}-t_{2n})}, \quad \forall t\in (t_{2n},S^{\epsilon}).$$
		Thus, we get
		\begin{equation}\label{lim}
			\max_{i=1,\dots,N} \langle x_{i}(t),v\rangle\leq M_{t_{2n}}+\epsilon-\epsilon e^{-K(S^{\epsilon}-t_{2n})}, \quad \forall t\in (t_{2n},S^{\epsilon}).
		\end{equation}
		Letting $t\to S^{\epsilon-}$ in \eqref{lim}, from \eqref{teps} we have that $$M_{t_{2n}}+\epsilon\leq M_{t_{2n}}+\epsilon-\epsilon e^{-K(S^{\epsilon}-t_{2n})}<M_{t_{2n}}+\epsilon,$$
		which is a contradiction. Thus, $S^{\epsilon}=t_{2n+1}$. As a consequence, we have that $$\max_{i=1,\dots,N}\langle x_{i}(t),v\rangle<M_{t_{2n}}+\epsilon, \quad \forall t\in (t_{2n},t_{2n+1}).$$
		From the arbitrariness of $\epsilon$ we can conclude that $$\max_{i=1,\dots,N}\langle x_{i}(t),v\rangle\leq M_{t_{2n}}, \quad \forall t\in [t_{2n},t_{2n+1}],$$
		from which $$\langle x_{i}(t),v\rangle\leq M_{t_{2n}}, \quad \forall t\in [t_{2n},t_{2n+1}], \,\forall i=1,\dots,N,$$
		which proves the second inequality in \eqref{scalpr}.
		\\Now, to prove the other inequality, let $v\in \RR^{d}$ and define $$m_{t_{2n}}=\min_{j=1,\dots,N}\langle x_{j}(t_{2n}),v\rangle.$$
		Then, for all $i=1,\dots,N$ and $t>t_{2n}$, by applying the second inequality in \eqref{scalpr} to the vector $-v\in\RR^{d}$ we get $$-\langle x_{i}(t),v\rangle=\langle x_{i}(t),-v\rangle\leq \max_{j=1,\dots,N}\langle x_{j}(t_{2n}),-v\rangle$$$$=-\min_{j=1,\dots,N}\langle x_{j}(t_{2n}),v\rangle=-m_{t_{2n}},$$
		from which $$\langle x_{j}(s),v\rangle\geq m_{t_{2n}}.$$
		Thus, also the first inequality in \eqref{scalpr} is fulfilled.
	\end{proof}
 \begin{lem}\label{L2}
		For each $n\in \mathbb{N}_{0}$ and $i,j=1,\dots,N$, we get \begin{equation}
			\label{dist}
			\lvert x_{i}(s)-x_{j}(t)\rvert\leq d(t_{2n}), \quad\forall s,t\in [t_{2n},t_{2n+1}].
		\end{equation} 
	\end{lem}
	\begin{proof}
		Fix $n\in\mathbb{N}_{0}$ and $i,j=1,\dots,N$. Given $s,t\in [t_{2n},t_{2n+1}]$, if $\lvert x_{i}(s)-x_{j}(t)\rvert=0$ then of course $d(t_{2n})\geq 0=\lvert x_{i}(s)-x_{j}(t)\rvert$. Thus, we can assume $\lvert x_{i}(s)-x_{j}(t)\rvert>0$ and we set $$v=\frac{x_{i}(s)-x_{j}(t)}{\lvert x_{i}(s)-x_{j}(t)\rvert}.$$
		It turns out that $v$ is a unit vector and, using \eqref{scalpr}, we can write 
		$$\begin{array}{l}
			\vspace{0.3cm}\displaystyle{\lvert x_{i}(s)-x_{j}(t)\rvert=\langle x_{i}(s)-x_{j}(t),v\rangle=\langle x_{i}(s),v\rangle-\langle x_{j}(t),v\rangle}\\
			\vspace{0.3cm}\displaystyle{\leq \max_{l=1,\dots,N}\langle x_{l}(t_{2n}),v\rangle-\min_{l=1,\dots,N}\langle x_{l}(t_{2n}),v\rangle}\\
			\vspace{0.3cm}\displaystyle{\leq \max_{l,k=1,\dots,N}\langle x_{l}(t_{2n})-x_{k}(t_{2n}),v\rangle}	\\
			\displaystyle{\leq \max_{l,k=1,\dots,N}\lvert x_{l}(t_{2n})-x_{k}(t_{2n})\rvert\lvert v\rvert=d(t_{2n}),}
		\end{array}$$
		which proves \eqref{dist}.
	\end{proof}
	\begin{oss}
		Let us note that from \eqref{dist}, in particular, it follows that
		\begin{equation}\label{dec}
			d(t_{2n+1})\leq d(t_{2n}),\quad \forall n\in \mathbb{N}_{0}.
		\end{equation}
		Indeed, let $i,j=1,\dots,N$ be such that $$d(t_{2n+1})=\lvert x_{i}(t_{2n+1})-x_{j}(t_{2n+1})\rvert.$$
		Then, if $d(t_{2n+1})=0$, of course we have $d(t_{2n+1})=0\leq d(t_{2n})$. So we can assume $d(t_{2n+1})>0$.
		Then, we can apply \eqref{dist} with $s,t=t_{2n+1}$ and we get $$d(t_{2n+1})=\lvert x_{i}(t_{2n+1})-x_{j}(t_{2n+1})\rvert\leq d(t_{2n}).$$
	\end{oss} 
	Now, we deal with the intervals in which the agents repeal each other. 	\begin{lem}\label{L1-1}
		For each $v\in \RR^{d}$ and $n\in \mathbb{N}_0,$  we have that 
		\begin{equation}\label{scalpr-1}
			\min_{j=1,\dots,N}\langle x_{j}(t_{2n+2}),v\rangle\leq \langle x_{i}(t),v\rangle\leq \max_{j=1,\dots,N}\langle x_{j}(t_{2n+2}),v\rangle,
		\end{equation}for all  $t\in [t_{2n+1},t_{2n+2}]$ and $i=1,\dots,N$. 
	\end{lem}
	\begin{proof}
		Let $n\in \mathbb{N}_{0}$. Given a vector $v\in \RR^{d}$, we set $$M_{t_{2n+2}}=\max_{j=1,\dots,N}\langle x_{j}(t_{2n+2}),v\rangle.$$
		For all $\epsilon >0$, let us define
		$$K^{\epsilon}:=\left\{t\in [t_{2n+1},t_{2n+2}] :\max_{i=1,\dots,N}\langle x_{i}(s),v\rangle < M_{t_{2n+2}}+\epsilon,\,\forall s\in (t,t_{2n+2}]\right\}.$$
		Denoted with $$S^{\epsilon}:=\inf K^{\epsilon},$$
		by continuity it holds that $S^{\epsilon}\in [t_{2n+1},t_{2n+2})$. \\We claim that $S^{\epsilon}=t_{2n+1}$. Indeed, suppose by contradiction that $S^{\epsilon}>t_{2n+1}$. Note that by definition of $S^{\epsilon}$, it turns out that \begin{equation}\label{max-1}
			\max_{i=1,\dots,N}\langle x_{i}(t),v\rangle<M_{t_{2n+2}}+\epsilon,\quad \forall t\in (S^{\epsilon},t_{2n+2}),
		\end{equation}
		and \begin{equation}\label{teps-1}
			\lim_{t\to S^{\epsilon+}}\max_{i=1,\dots,N}\langle x_{i}(t),v\rangle=M_{t_{2n+2}}+\epsilon.
		\end{equation}
		For all $i=1,\dots,N$ and $t\in (S^{\epsilon},t_{2n+2})$, we compute
		$$\frac{d}{dt}\langle x_{i}(t),v\rangle=\frac{1}{N-1}\sum_{j:j\neq i}\alpha(t)\psi(x_{i}(t), x_{j}(t))\langle x_{j}(t)-x_{i}(t),v\rangle.$$
		Notice that, being $t\in (S^{\epsilon},t_{2n+2})$, then  $t\in (t_{2n+1}, t_{2n+2})$ and, as a consequence, it holds that $\alpha(t)=-1.$ Thus, using \eqref{K} and \eqref{max-1}, we can write $$\frac{d}{dt}\langle x_{i}(t),v\rangle\geq -\frac{1}{N-1}\sum_{j:j\neq i}\psi(x_{i}(t), x_{j}(t))(M_{t_{2n+2}}+\epsilon-\langle x_{i}(t),v\rangle)$$$$\geq -K(M_{t_{2n+2}}+\epsilon-\langle x_{i}(t),v\rangle).$$
		Then, from Gronwall's inequality, we get
		$$\begin{array}{l}
			\vspace{0.2cm}\displaystyle{
				\langle x_{i}(t),v\rangle\geq e^{K(t-S^{\epsilon})}\langle x_{i}(S^{\epsilon}),v\rangle-K(M_{t_{2n+2}}+\epsilon)\int_{S^{\epsilon}}^{t}e^{K(t-s)}ds}\\
			\vspace{0.3cm}\displaystyle{\hspace{1.7 cm}
				=e^{K(t-S^{\epsilon})}\langle x_{i}(S^{\epsilon}),v\rangle+(M_{t_{2n+2}}+\epsilon)e^{Kt}(e^{-Kt}-e^{-KS^{\epsilon}})}\\
			\displaystyle{\hspace{1.7 cm}
				=e^{K(t-S^{\epsilon})}\langle x_{i}(S^{\epsilon}),v\rangle+(M_{t_{2n+2}}+\epsilon)(1-e^{K(t-S^{\epsilon})}),}\\
		\end{array}
		$$for all $t\in (S^{\epsilon}, t_{2n+2})$ and $i=1,\dots,N$. Taking the supremum for $i=1,\dots,N$, we get 
		$$\max_{i=1,\dots,N}\langle x_{i}(t),v\rangle\geq  e^{K(t-S^{\epsilon})}\max_{i=1,\dots,N}\langle x_{i}(S^{\epsilon}),v\rangle+(M_{t_{2n+2}}+\epsilon)(1-e^{K(t-S^{\epsilon})})$$$$=(M_{t_{2n+2}}+\epsilon)e^{K(t-S^{\epsilon})}+(M_{t_{2n+2}}+\epsilon)(1-e^{K(t-S^{\epsilon})})=M_{t_{2n+2}}+\epsilon.$$
		For $t= t_{2n+2}$, we have that $$M_{t_{2n+2}}=\max_{i=1,\dots,N}\langle x_{i}(t_{2n+2}),v\rangle \geq M_{t_{2n+2}}+\epsilon>M_{t_{2n+2}},$$
		which is a contradiction. Thus, $S^{\epsilon}=t_{2n+1}$. As a consequence, we have that $$\max_{i=1,\dots,N}\langle x_{i}(t),v\rangle<M_{t_{2n+2}}+\epsilon, \quad \forall t\in (t_{2n+1},t_{2n+2}).$$
		From the arbitrariness of $\epsilon$ we can conclude that $$\max_{i=1,\dots,N}\langle x_{i}(t),v\rangle\leq M_{t_{2n+2}}, \quad \forall t\in [t_{2n+1},t_{2n+2}],$$
		from which $$\langle x_{i}(t),v\rangle\leq M_{t_{2n+2}}, \quad \forall t\in [t_{2n+1},t_{2n+2}], \,\forall i=1,\dots,N.$$
		Thus, the second inequality in \eqref{scalpr-1} holds for all $t\in [t_{2n+1},t_{2n+2}]$.
		\\Now, to prove the other inequality, let $v\in \RR^{d}$ and define $$m_{t_{2n+2}}=\min_{j=1,\dots,N}\langle x_{j}(t_{2n+2}),v\rangle.$$
		Then, for all $i=1,\dots,N$ and $t\in[t_{2n+1},t_{2n+2}]$, by applying the second inequality in \eqref{scalpr-1} to the vector $-v\in\RR^{d}$ we get $$-\langle x_{i}(t),v\rangle=\langle x_{i}(t),-v\rangle\leq \max_{j=1,\dots,N}\langle x_{j}(t_{2n+2}),-v\rangle$$$$=-\min_{j=1,\dots,N}\langle x_{j}(t_{2n+2}),v\rangle=-m_{t_{2n+2}},$$
		from which $$\langle x_{j}(s),v\rangle\geq m_{t_{2n+2}}.$$
		Thus, also the first inequality in \eqref{scalpr-1} is fulfilled.
	\end{proof}
	\begin{lem}\label{L2-1}
		For each $n\in \mathbb{N}_{0}$ and $i,j=1,\dots,N$, we get \begin{equation}
			\label{dist-1}
			\lvert x_{i}(s)-x_{j}(t)\rvert\leq d(t_{2n+2}), \quad\forall s,t\in [t_{2n+1},t_{2n+2}].
		\end{equation} 
	\end{lem}
	\begin{proof}
		Fix $n\in\mathbb{N}_{0}$ and $i,j=1,\dots,N$. Given $s,t\in [t_{2n+1},t_{2n+2}]$, if $\lvert x_{i}(s)-x_{j}(t)\rvert=0$ then of course $d(t_{2n+2})\geq 0=\lvert x_{i}(s)-x_{j}(t)\rvert$. Thus, we can assume $\lvert x_{i}(s)-x_{j}(t)\rvert>0$ and we set $$v=\frac{x_{i}(s)-x_{j}(t)}{\lvert x_{i}(s)-x_{j}(t)\rvert}.$$
		It turns out that $v$ is a unit vector and, using \eqref{scalpr}, we can write 
		$$\begin{array}{l}
			\vspace{0.3cm}\displaystyle{\lvert x_{i}(s)-x_{j}(t)\rvert=\langle x_{i}(s)-x_{j}(t),v\rangle=\langle x_{i}(s),v\rangle-\langle x_{j}(t),v\rangle}\\
			\vspace{0.3cm}\displaystyle{\leq \max_{l=1,\dots,N}\langle x_{l}(t_{2n+2}),v\rangle-\min_{l=1,\dots,N}\langle x_{l}(t_{2n+2}),v\rangle}\\
			\vspace{0.3cm}\displaystyle{\leq \max_{l,k=1,\dots,N}\langle x_{l}(t_{2n+2})-x_{k}(t_{2n+2}),v\rangle}	\\
			\displaystyle{\leq \max_{l,k=1,\dots,N}\lvert x_{l}(t_{2n+2})-x_{k}(t_{2n+2})\rvert\lvert v\rvert=d(t_{2n+2}),}
		\end{array}$$
		which proves \eqref{dist-1}.
	\end{proof}
	\begin{oss}
		Let us note that from \eqref{dist}, in particular, it follows that
		\begin{equation}\label{cre}
			d(t_{2n+2})\geq d(t_{2n+1}),\quad \forall n\in \mathbb{N}_{0}.
		\end{equation}
		Indeed, let $i,j=1,\dots,N$ be such that $$d(t_{2n+1})=\lvert x_{i}(t_{2n+1})-x_{j}(t_{2n+1})\rvert.$$
		Then, if $d(t_{2n+1})=0$, of course we have $d(t_{2n+2})=0\leq d(t_{2n+1})$. So we can assume $d(t_{2n+1})>0$.
		Then, we can apply \eqref{dist-1} with $s,t=t_{2n+1}$ and we get $$d(t_{2n+1})=\lvert x_{i}(t_{2n+1})-x_{j}(t_{2n+1})\rvert\leq d(t_{2n+2}).$$
	\end{oss}
Now, with a quite simple argument, it can be proved that, in the intervals in which the particles attract each other, the solutions of the system under consideration have a bound that is uniform with respect to $i=1,\dots,N$, but that depends of the maximum value assumed by the opinions of the agents at the left end of the good interval. \begin{lem}\label{L3}
		For every $i=1,\dots,N,$ we have that \begin{equation}\label{boundsol}
			\lvert x_{i}(t)\rvert\leq M_{2n}^{0},\quad \forall t\in [t_{2n},t_{2n+1}],
		\end{equation}
		where $M_{2n}^{0}$ is that in \eqref{M0}.	
	\end{lem}
	\begin{proof}
		Given $i=1,\dots,N$ and $t\in [t_{2n},t_{2n+1}]$, if $\lvert x_{i}(t)\rvert =0$ then trivially $M_{2n}^{0}\geq 0=\lvert x_{i}(t)\rvert $. On the contrary, if $\lvert x_{i}(t)\rvert >0$, we define $$v=\frac{x_{i}(t)}{\lvert x_{i}(t)\rvert},$$
		which is a unit vector for which we can write
		$$\lvert x_{i}(t)\rvert=\langle x_{i}(t),v\rangle. $$
		Then, using \eqref{scalpr} and the Cauchy-Schwarz inequality we get $$\lvert x_{i}(t)\rvert\leq \max_{j=1,\dots,N}\langle x_{j}(t_{2n}),v\rangle\leq \max_{j=1,\dots,N}\lvert x_{j}(t_{2n})\rvert\lvert v\rvert$$$$=\max_{j=1,\dots,N}\lvert x_{j}(t_{2n})\rvert=M_{2n}^{0},$$
		which proves \eqref{boundsol}.
	\end{proof}
Now, we find a bound from below for the influence function $\psi$. This will be crucial for the proof of the asymptotic consensus.
\begin{prop}\label{14nov}
	Assume \eqref{bad}. Then, for all $t\geq 0$, we have that
	\begin{equation}\label{min}
		\psi(x_{i}(t),x_{j}(t))\geq \psi_{0},\quad \forall i,j=1,\dots,N,
	\end{equation}
where $\psi_0$ is the positive constant in \eqref{psi0}.
\end{prop}
\begin{oss}
	Let us note that the previous result holds in particular under assumption \eqref{sum1}. Indeed, we have already noticed that the condition \eqref{sum1} implies \eqref{bad}.
\end{oss}
\begin{proof}[Proof of Proposition \ref{14nov}] 
	From \eqref{boundsol}, it follows that
	\begin{equation}\label{unifboundsol}
		\max_{i=1,\dots,N}\lvert x_{i}(t)\rvert\leq M^{0},\quad \forall t\geq 0.
	\end{equation}
To see this, fix $t\geq 0$. Then, there exists $n\in \mathbb{N}_{0}$ such that $t\in [t_{2n},t_{2n+2}]$. Thus, if $t\in [t_{2n},t_{2n+1}]$, from \eqref{boundsol} we have that
\begin{equation}\label{boundsol1}
	\lvert x_{i}(t)\lvert\leq M^{0}_{2n}=\max_{i=1,\dots,N}\lvert x_{i}(t_{2n})\lvert, \quad\forall i=1,\dots,N.
\end{equation}
On the other hand, assume that $t\in (t_{2n+1},t_{2n+2})$. Given $i=1,\dots,N$, if $\lvert x_{i}(t)\lvert>0$, we define the unit vector $$v=\frac{x_{i}(t)}{\lvert x_{i}(t)\lvert}.$$
Then, $$\lvert x_{i}(t)\lvert=\langle  x_{i}(t),v\rangle.$$
Now, for all $s\in [t_{2n+1},t)$, it holds that
$$\frac{d}{ds}\langle  x_{i}(s),v\rangle=-\frac{1}{N-1}\sum_{j:j\neq i}\psi(x_{i}(s),x_{j}(s))(\langle  x_{j}(s),v\rangle-\langle  x_{i}(s),v\rangle)$$
$$=\frac{1}{N-1}\sum_{j:j\neq i}\psi(x_{i}(s),x_{j}(s))(\langle  x_{i}(s),v\rangle-\langle  x_{j}(s),v\rangle).$$
Thus, denoted with $$m_{t_{2n+2}}=\min_{l=1,\dots,N}\langle x_{l}(t_{2n+2}),v\rangle,$$
the first inequality in \eqref{scalpr-1} implies that
$$\langle x_{l}(s),v\rangle\geq m_{t_{2n+2}}, \quad \forall s\in [t_{2n+1},t],\, \forall l=1,\dots,N.$$
As a consequence, we get
$$\frac{d}{ds}\langle  x_{i}(s),v\rangle\leq \frac{1}{N-1}\sum_{j:j\neq i}\psi(x_{i}(s),x_{j}(s))(\langle  x_{i}(s),v\rangle-m_{t_{2n+2}})$$
$$\leq K(\langle  x_{i}(s),v\rangle-m_{t_{2n+2}}).$$
So, the Gronwall's inequality yields
$$\langle  x_{i}(s),v\rangle\leq e^{K(s-t_{2n+1})}\langle  x_{i}(t_{2n+1}),v\rangle-Km_{t_{2n+2}}\int_{t_{2n+1}}^{s}e^{K(s-r)}dr$$
$$=e^{K(s-t_{2n+1})}\langle  x_{i}(t_{2n+1}),v\rangle+m_{t_{2n+2}}(1-e^{K(s-t_{2n+1})}),$$
for all $s\in [t_{2n+1},t]$. In particular, for $s=t$ it comes that
$$\begin{array}{l}
	\vspace{0.3cm}\displaystyle{\langle  x_{i}(t),v\rangle\leq e^{K(t-t_{2n+1})}\langle  x_{i}(t_{2n+1}),v\rangle+m_{t_{2n+2}}(1-e^{K(t-t_{2n+1})})}\\
	\vspace{0.3cm}\displaystyle{\hspace{1.6cm}=e^{K(t-t_{2n+1})}(\langle  x_{i}(t_{2n+1}),v\rangle-m_{t_{2n+2}})+m_{t_{2n+2}}}\\
	\vspace{0.3cm}\displaystyle{\hspace{1.6cm}\leq e^{K(t_{2n+2}-t_{2n+1})}(\langle  x_{i}(t_{2n+1}),v\rangle-m_{t_{2n+2}})+m_{t_{2n+2}}}\\
	\vspace{0.3cm}\displaystyle{\hspace{1.6cm}=e^{K(t_{2n+2}-t_{2n+1})}\langle  x_{i}(t_{2n+1}),v\rangle+m_{t_{2n+2}}(1-e^{K(t_{2n+2}-t_{2n+1})})}\\
	\vspace{0.3cm}\displaystyle{\hspace{1.6cm}\leq e^{K(t_{2n+2}-t_{2n+1})}\langle  x_{i}(t_{2n+1}),v\rangle}\\
	\displaystyle{\hspace{1.6cm}\leq e^{K(t_{2n+2}-t_{2n+1})}  \lvert x_{i}(t_{2n+1})\rvert.}
\end{array}$$
Thus, using \eqref{boundsol} we get $$\lvert x_{i}(t)\rvert=\langle  x_{i}(t),v\rangle\leq e^{K(t_{2n+2}-t_{2n+1})}M^{0}_{2n}.$$
Of course, the above inequality is satisfied also if $\lvert x_{i}(t)\rvert=0$. Thus,  combining this last inequality with \eqref{boundsol1}, being $e^{K(t_{2n+2}-t_{2n+1})}>1$, we can conclude that
\begin{equation}\label{boundsol2}
	\lvert x_{i}(t)\rvert\leq e^{K(t_{2n+2}-t_{2n+1})}M^{0}_{2n},\quad \forall n\in \mathbb{N}_{0},\, t\in [t_{2n},t_{2n+2}],\,i=1,\dots,N.
\end{equation}
Now, let us note that, using an induction argument, from \eqref{boundsol2} it follows that
$$M^{0}_{2n+2}=\max_{i=1,\dots,N}\lvert x_{i}(t_{2n+2})\rvert\leq M^{0}_{0}\prod_{p=0}^{n}e^{K(t_{2p+2}-t_{2p+1})},\quad \forall n\geq 0.$$
As a consequence, for all $t\geq 0$, it holds
$$\lvert x_{i}(t)\rvert\leq M^{0}_{0}\prod_{p=0}^{n}e^{K(t_{2p+2}-t_{2p+1})}=e^{K\sum_{p=0}^{n}(t_{2p+2}-t_{2p+1})}M^{0}_{0}.$$
Then, for all $t\geq 0$, $$\lvert x_{i}(t)\rvert\leq M^{0}_{0}e^{K\sum_{p=0}^{\infty}(t_{2p+2}-t_{2p+1})},$$
which proves \eqref{unifboundsol}.
\\Finally, from \eqref{unifboundsol}, we deduce that, for all $t\geq 0$
$$\psi(x_{i}(t),x_{i}(t))\geq \psi_{0}.$$
\end{proof}	
\begin{prop}\label{L4}
		For all $i,j=1,\dots,N$,  unit vector $v\in \RR^{d}$ and $n\in\mathbb{N}_{0}$ we have that 
		\begin{equation}\label{4}
			\langle x_{i}(t)-x_{j}(t),v\rangle\leq e^{-K(t-\bar{t})}\langle x_{i}(\bar{t})-x_{j}(\bar{t}),v\rangle+(1-e^{-K(t-\bar{t})})d(t_{2n}),
		\end{equation}
		for all $t_{2n+1}\geq t\geq \bar{t}\geq t_{2n}$. 
	\end{prop}
	\begin{proof}
		Fix $n\in\mathbb{N}_{0}$ and $v\in\RR^{d}$ such that $\lvert v\rvert=1$. We set $$M_{t_{2n}}=\max_{i=1,\dots,N}\langle x_{i}(t_{2n}),v\rangle,$$$$m_{t_{2n}}=\min_{i=1,\dots,N}\langle x_{i}(t_{2n}),v\rangle.$$
		Then, it is easy to see that $M_{t_{2n}}-m_{t_{2n}}\leq d(t_{2n})$. Now, for all $i=1,\dots,N$ and $t_{2n+1}>t\geq \bar{t}\geq t_{2n}$, from \eqref{scalpr} we have that 
		$$
		\begin{array}{l}
			\displaystyle{
				\frac{d}{dt}\langle x_{i}(t),v\rangle=\sum_{j:j\neq i}\alpha(t)a_{ij}(t)\langle x_{j}(t)-x_{i}(t),v\rangle}\\
			\displaystyle{\hspace{2 cm}=\frac{1}{N-1}\sum_{j:j\neq i}\psi( x_{i}(t), x_{j}(t))(\langle x_{j}(t),v\rangle-\langle x_{i}(t),v\rangle)}\\
			\displaystyle{\hspace{2 cm}\leq \frac{1}{N-1}\sum_{j:j\neq i}\psi(x_{i}(t), x_{j}(t))(M_{t_{2n}}-\langle x_{i}(t),v\rangle).}
		\end{array}
		$$
		Note that $\langle x_{i}(t),v\rangle\leq M_{t_{2n}}$ from \eqref{scalpr}, so that $M_{t_{2n}}-\langle x_{i}(t),v\rangle\geq 0$. Therefore, using \eqref{K}, we can write $$\frac{d}{dt}\langle x_{i}(t),v\rangle\leq\frac{1}{N-1}K\sum_{j:j\neq i}(M_{t_{2n}}-\langle x_{i}(t),v\rangle)=K(M_{t_{2n}}-\langle x_{i}(t),v\rangle).$$
		Thus, from the Gronwall's inequality it comes that $$\langle x_{i}(t),v\rangle\leq e^{-K(t-\bar{t})}\langle x_{i}(\bar{t}),v\rangle+\int_{\bar{t}}^{t}KM_{t_{2n}}e^{-K(t-\bar{t})+K(s-\bar{t})}ds$$
		$$\hspace{1.3cm}=e^{-K(t-\bar{t})}\langle x_{i}(\bar{t}),v\rangle+e^{-K(t-\bar{t})}M_{t_{2n}}(e^{K(t-\bar{t})}-1),$$
		that is \begin{equation}\label{2}
			\langle x_{i}(t),v\rangle\leq e^{-K(t-\bar{t})}\langle x_{i}(\bar{t}),v\rangle+(1-e^{-K(t-\bar{t})})M_{t_{2n}}.
		\end{equation}
		Now, for all $i=1,\dots,N$ and $t_{2n+1}>t\geq \bar{t}\geq t_{2n}$, using \eqref{scalpr} it holds
		$$
		\begin{array}{l}
			\displaystyle{
				\frac{d}{dt}\langle x_{i}(t),v\rangle=\frac{1}{N-1}\sum_{i:j\neq i}\alpha(t)\psi(x_{i}(t),x_{j}(t)(\langle x_{j}(t),v\rangle-\langle x_{i}(t),v\rangle)}\\
			\displaystyle{
				\hspace{2 cm}\geq \frac{1}{N-1}\sum_{j:j\neq i}\psi(x_{i}(t),x_{j}(t))(m_{t_{2n}}-\langle x_{i}(t),v\rangle).}
		\end{array}
		$$
		Then, since $\langle x_{i}(t),v\rangle\geq m_{t_{2n}}$ from \eqref{scalpr} and by recalling that $\psi$ is bounded from \eqref{K}, we get $$\frac{d}{dt}\langle x_{i}(t),v\rangle\geq K(m_{t_{2n}}-\langle x_{i}(t),v\rangle).$$
		Thus, by using the Gronwall's inequality, it turns out that
		\begin{equation}\label{3}
				\langle x_{i}(t),v\rangle\geq e^{-K(t-\bar{t})}\langle x_{i}(\bar{t}),v\rangle+(1-e^{-K(t-\bar{t})})m_{t_{2n}}.
		\end{equation}
		Therefore, for all $i,j=1,\dots,N$ and $t_{2n+1}>t\geq \bar{t}\geq t_{2n}$, by using \eqref{2} and \eqref{3} and by recalling that $M_{t_{2n}}-m_{t_{2n}}\leq d(t_{2n}),$ we finally get 
		$$
		\begin{array}{l}
			\vspace{0.3cm}\displaystyle{\langle x_{i}(t)-x_{j}(t),v\rangle=\langle x_{i}(t),v\rangle-\langle x_{j}(t),v\rangle}\\
			\vspace{0.3cm}\displaystyle{\hspace{2.8 cm}\leq e^{-K(t-\bar{t})}\langle x_{i}(\bar{t}),v\rangle+(1-e^{-K(t-\bar{t})})M_{t_{2n}}}\\
			\vspace{0.3cm}\displaystyle{\hspace{3.1 cm}-e^{-K(t-\bar{t})}\langle x_{j}(\bar{t}),v\rangle-(1-e^{-K(t-\bar{t})})m_{t_{2n}}}\\
			\vspace{0.3cm}\displaystyle{\hspace{2.8 cm}=e^{-K(t-\bar{t})}\langle x_{i}(\bar{t})-x_{j}(\bar{t}),v\rangle+(1-e^{-K(t-\bar{t})})(M_{t_{2n}}-m_{t_{2n}})}\\
			\displaystyle{\hspace{2.8 cm}\leq e^{-K(t-\bar{t})}\langle x_{i}(\bar{t})-x_{j}(\bar{t}),v\rangle+(1-e^{-K(t-\bar{t})})d(t_{2n}),}
		\end{array}
		$$
		i.e. \eqref{4} holds true.
	\end{proof}
	Thanks to the presence of a uniform bound on the influence function $\psi$, the following fundamental estimate holds in the intervals of positive interaction.
	\begin{prop}\label{L5}
		For all $n\in \mathbb{N}_{0}$, there exists a constant $C_{2n}\in (0,1),$ independent of $N\in\mathbb{N}_0,$ such that
		\begin{equation}\label{n-2}
			d(t_{2n+1})\leq C_{2n} d(t_{2n}).
		\end{equation}
	\end{prop}
	\begin{proof}
		Let $n\in \mathbb{N}_{0}$. Trivially, if $d(t_{2n+1})=0$, then of course inequality \eqref{n-2} holds for any positive constant. So, suppose $d(t_{2n+1})>0$. Let $i,j=1,\dots,N$ be such that $d(t_{2n+1})=\lvert x_{i}(t_{2n+1})-x_{j}(t_{2n+1})\rvert$. We set $$v=\frac{x_{i}(t_{2n+1})-x_{j}(t_{2n+1})}{\lvert x_{i}(t_{2n+1})-x_{j}(t_{2n+1})\rvert}.$$
		Then, $v$ is a unit vector for which we can write $$d(t_{2n+1})=\langle x_{i}(t_{2n+1})-x_{j}(t_{2n+1}),v\rangle.$$
		Let us define $$M_{t_{2n}}=\max_{l=1,\dots,N} x_{l}(t_{2n}),v\rangle,$$
		$$m_{t_{2n}}=\min_{l=1,\dots,N}\langle x_{l}(t_{2n}),v\rangle.$$
		Then $M_{t_{2n}}-m_{t_{2n}}\leq d(t_{2n})$. 
		\\Now, we distinguish two different situations.
		\par\textit{Case I.} Assume that there exists $\bar{t}\in [t_{2n},t_{2n+1})$ such that 
		$$\langle x_{i}(\bar{t})-x_{j}(\bar{t}),v\rangle<0.$$
		Then, from \eqref{4} with $t_{2n+1}\geq \bar{t}\geq t_{2n}$, we have \begin{equation}\label{t0}
			\begin{split}
				d(t_{2n+1})&\leq e^{-K(t_{2n+1}-\bar{t})}\langle x_{i}(\bar{t})-x_{j}(\bar{t}),v\rangle+(1-e^{-K(t_{2n+1}-\bar{t})})d(t_{2n})\\&\leq (1-e^{-K(t_{2n+1}-\bar{t})})d(t_{2n})\\&\leq (1-e^{-K(t_{2n+1}-t_{2n})})d(t_{2n}).
			\end{split}
		\end{equation}
		\par\textit{Case II.} Assume it rather holds \begin{equation}\label{pos}
			\langle x_{i}(t)-x_{j}(t),v\rangle\geq 0,\quad \forall t\in [t_{2n},t_{2n+1}).
		\end{equation}
		Then, for every  $t\in [t_{2n},t_{2n+1})$, we have that 
		$$
		\begin{array}{l}
			\displaystyle{
				\frac{d}{dt}\langle x_{i}(t)-x_{j}(t),v\rangle=\frac{1}{N-1}\sum_{l:l\neq i}\psi(x_{i}(t),x_{l}(t))\langle x_{l}(t)-x_{i}(t),v\rangle}\\
			\displaystyle{\hspace{1.1cm}-\frac{1}{N-1}\sum_{l:l\neq j}\psi(x_{i}(t),x_{l}(t))\langle x_{l}(t)-x_{j}(t),v\rangle}\\
			\displaystyle{\hspace{0.6cm}=\frac{1}{N-1}\sum_{l:l\neq i}\psi(x_{i}(t),x_{l}(t))(\langle x_{l}(t),v\rangle-M_{t_{2n}}+M_{t_{2n}}-\langle x_{i}(t),v\rangle)}\\
			\displaystyle{\hspace{1.1cm}+\frac{1}{N-1}\sum_{l:l\neq j}\psi(x_{i}(t),x_{l}(t))(\langle x_{j}(t),v\rangle-m_{t_{2n}}+m_{t_{2n}}-\langle x_{l}(t),v\rangle)}\\
			\displaystyle{\hspace{5.5 cm}:=S_1+S_2.}
		\end{array}
		$$
		Now, being $t\in [t_{2n},t_{2n+1})$,  from \eqref{scalpr} we have that
		\begin{equation}\label{7}
			m_{t_{2n}}\leq\langle x_{k}(t),v\rangle\leq M_{t_{2n}}, \quad\forall k=1,\dots,N.
		\end{equation}
		Therefore, using \eqref{boundsol}, we get
		$$
		\begin{array}{l}
			\displaystyle{
				S_1=\frac{1}{N-1}\sum_{l:l\neq i}\psi(x_{i}(t),x_{l}(t-))(\langle x_{l}(t),v\rangle-M_{t_{2n}})}\\
			\displaystyle{\hspace{0.7 cm}+\frac{1}{N-1}\sum_{l:l\neq i}\psi(x_{i}(t),x_{l}(t))(M_{t_{2n}}-\langle x_{i}(t),v\rangle)}\\
			\displaystyle{\hspace{0.4 cm}\leq \frac{1}{N-1}\psi_{0}\sum_{l:l\neq i}(\langle x_{l}(t),v\rangle-M_{t_{2n}})+K(M_{t_{2n}}-\langle x_{i}(t),v\rangle)},
		\end{array}
		$$
		and	
		$$
		\begin{array}{l}
			\displaystyle{
				S_2=\frac{1}{N-1}\sum_{l:l\neq j}\psi(x_{i}(t),x_{l}(t))(\langle x_{j}(t),v\rangle-m_{t_{2n}})}\\
			\displaystyle{\hspace{0.7 cm}+\frac{1}{N-1}\sum_{l:l\neq j}\psi(x_{i}(t),x_{l}(t))(m_{t_{2n}}-\langle x_{l}(t),v\rangle)}\\
			\displaystyle{\hspace{0.4 cm}\leq K(\langle x_{j}(t),v\rangle-m_{t_{2n}})+\frac{1}{N-1}\psi_{0}\sum_{l:l\neq j}(m_{t_{2n}}-\langle x_{l}(t),v\rangle).}
		\end{array}
		$$
		Combining this last fact with \eqref{7} it comes that 
		$$\begin{array}{l}
			\vspace{0.2cm}\displaystyle{\frac{d}{dt}\langle x_{i}(t)-x_{j}(t),v\rangle\leq K(M_{t_{2n}}-m_{t_{2n}}-\langle x_{i}(t)-x_{j}(t),v\rangle)}\\
			\vspace{0.2cm}\displaystyle{\hspace{1.8 cm}+\frac{1}{N-1}\psi_{0}\sum_{l:l\neq i,j}(\langle x_{l}(t),v\rangle-M_{t_{2n}}+m_{t_{2n}}-\langle x_{l}(t),v\rangle)}\\
			\vspace{0.2cm}\displaystyle{\hspace{1.8 cm}+\frac{1}{N-1}\psi_{0}(\langle x_{j}(t),v\rangle-M_{t_{2n}}+m_{t_{2n}}-\langle x_{i}(t),v\rangle)}\\
			\vspace{0.2cm}\displaystyle{\hspace{1.5 cm}=K(M_{t_{2n}}-m_{t_{2n}})-K\langle x_{i}(t)-x_{j}(t),v\rangle+\frac{N-2}{N-1}\psi_{0}(-M_{t_{2n}}+m_{t_{2n}})}\\
			\vspace{0.2cm}\displaystyle{\hspace{1.8 cm}+\frac{1}{N-1}\psi_{0}(\langle x_{j}(t),v\rangle-M_{t_{2n}}+m_{t_{2n}}-\langle x_{i}(t_{2n}),v\rangle)}.
		\end{array}
		$$
		Now, from \eqref{pos} we get
		$$\begin{array}{l}
			\vspace{0.2cm}\displaystyle{
				\frac{d}{dt}\langle x_{i}(t)-x_{j}(t),v\rangle\leq K(M_{t_{2n}}-m_{t_{2n}})-K\langle x_{i}(t)-x_{j}(t),v\rangle
			}\\
			\vspace{0.4 cm}\displaystyle{\hspace{1.8 cm}
				+\frac{N-2}{N-1}\psi_{0}(-M_{t_{2n}}+m_{t_{2n}})+\frac{1}{N-1}\psi_{0}(-M_{t_{2n}}+m_{t_{2n}})
			}\\
			\vspace{0.4 cm}\displaystyle{\hspace{1.8 cm}
				-\frac{1}{N-1}\psi_{0}\langle x_{i}(t)-x_{j}(t),v\rangle
			}\\
			\vspace{0.3 cm}\displaystyle{\hspace{1.5 cm}\leq K(M_{t_{2n}}-m_{t_{2n}})-K\langle x_{i}(t)-x_{j}(t),v\rangle+\psi_{0}(-M_{t_{2n}}+m_{t_{2n}})
			}\\
			\displaystyle{\hspace{1.5 cm}=\left(K-\psi_{0}\right)(M_{t_{2n}}-m_{t_{2n}})-K\langle x_{i}(t)-x_{j}(t),v\rangle.}
		\end{array}$$
		Hence, from Gronwall's inequality it comes that $$\langle x_{i}(t)-x_{j}(t),v\rangle \leq e^{-K(t-t_{2n})}\langle x_{i}(t_{2n})-x_{j}(t_{2n}),v\rangle$$$$+(M_{t_{2n}}-m_{t_{2n}})\int_{t_{2n}}^{t}\left(K-\psi_{0}\right)e^{-K(t-s)}ds,$$
		for all $t\in [t_{2n},t_{2n+1})$. In particular, for $t= t_{2n+1}$, from \eqref{dist} it comes that 
		$$\begin{array}{l}
			\vspace{0.2cm}\displaystyle{d(t_{2n+1})\leq e^{-K(t_{2n+1}-t_{2n})}\langle x_{i}(t_{2n})-x_{j}(t_{2n}),v\rangle+(M_{t_{2n}}-m_{t_{2n}})\int_{t_{2n}}^{t_{2n+1}}(K-\psi_{0})e^{-K(t_{2n+1}-s)}ds}\\
			\vspace{0.2cm}\displaystyle{\hspace{1.3cm}\leq e^{-K(t_{2n+1}-t_{2n})}\lvert x_{i}(t_{2n})-x_{j}(t_{2n})\rvert +(M_{t_{2n}}-m_{t_{2n}})\int_{t_{2n}}^{t_{2n+1}}(K-\psi_{0})e^{-K(t_{2n+1}-s)}ds}\\
			\vspace{0.2cm}\displaystyle{\hspace{1.3cm}\leq \left(e^{-K(t_{2n+1}-t_{2n})} +K\int_{t_{2n}}^{t_{2n+1}}e^{-K(t_{2n+1}-s)}ds-\psi_{0}\int_{t_{2n}}^{t_{2n+1}}e^{-K(t_{2n+1}-s)}ds\right)d(t_{2n})}\\
			\vspace{0.2cm}\displaystyle{\hspace{1.3cm}= \left(e^{-K(t_{2n+1}-t_{2n})} +1-e^{-K(t_{2n+1}-t_{2n})}-\frac{\psi_{0}}{K}(1-e^{-K(t_{2n+1}-t_{2n})})\right)d(t_{2n})}\\
			\vspace{0.2cm}\displaystyle{\hspace{1.3cm}=\left(1-\frac{\psi_{0}}{K}(1-e^{-K(t_{2n+1}-t_{2n})})\right)d(t_{2n}).}
		\end{array}$$
		So, if we set \begin{equation}\label{cn}
			C_{2n}:=\max\left\{1-e^{-K(t_{2n+1}-t_{2n})},1-\frac{\psi_{0}}{K}(1-e^{-K(t_{2n+1}-t_{2n})})\right\},
		\end{equation}
		$C_{2n}\in (0,1)$ is the constant for which \eqref{n-2} holds.
	\end{proof}
In the bad intervals, the previous estimate is not valid, being the diameter of the solution evaluated at $t_{2n+2}$ larger than the diameter evaluated at $t_{2n+1}$. However, the growth of the diameter in the intervals with negative interaction can be in some way controlled, as the following lemma shows.\begin{lem}
	For all $n\in \mathbb{N}_{0}$, we have that
	\begin{equation}\label{boundgrowth}
		d(t_{2n+2})\leq \frac{e^{K(t_{2n+2}-t_{2n+1})}}{2-e^{K(t_{2n+2}-t_{2n+1})}}d(t_{2n+1}).
	\end{equation}
\end{lem} 
\begin{proof}
	Let $n\in \mathbb{N}_{0}$. Let all $i,j=1,\dots,N$ be such that $$d(t_{2n+2})=\lvert x_{i}(t_{2n+2})-x_{j}(t_{2n+2})\rvert.$$
	If $d(t_{2n+2})=0$, then from \eqref{cre} also $d(t_{2n+1})=0$ and of course inequality \eqref{boundgrowth} is fulfilled. So, we can assume that $d(t_{2n+2})>0$. In this case, let $v\in \mathbb{R}^{d}$ be the so defined unit vector
	$$v=\frac{x_{i}(t_{2n+2})-x_{j}(t_{2n+2})}{\lvert x_{i}(t_{2n+2})-x_{j}(t_{2n+2})\rvert}.$$
	Then, $$d(t_{2n+2})=\langle x_{i}(t_{2n+2})-x_{j}(t_{2n+2}),v\rangle .$$
	Moreover, we set
	$$M_{t_{2n+2}}=\max_{k=1,\dots,N}\langle x_{k}(t_{2n+2}),v\rangle,$$
	$$m_{t_{2n+2}}=\min_{k=1,\dots,N}\langle x_{k}(t_{2n+2}),v\rangle.$$
	Thus, from \eqref{scalpr-1}, for all $t\in [t_{2n+1},t_{2n+2}],$ it holds
	\begin{equation}\label{dis}
		m_{t_{2n+2}}\leq\langle x_{k}(t),v\rangle\leq M_{t_{2n+2}},\quad \forall k=1,\dots,N.
	\end{equation}
Now, for all $t\in [t_{2n+1},t_{2n+2})$, using the first inequality in \eqref{dis} and the fact that $\alpha(t)=-1$, we have that
	$$\begin{array}{l}
		\vspace{0.3cm}\displaystyle{\frac{d}{dt}\langle x_{i}(t),v\rangle=-\frac{1}{N-1}\sum_{l:l\neq i}\psi(x_{i}(t),x_{l}(t))(\langle x_{l}(t),v\rangle-\langle x_{i}(t),v\rangle)}\\
		\vspace{0.3cm}\displaystyle{\hspace{2cm}\leq -\frac{1}{N-1}\sum_{l:l\neq i}\psi(x_{i}(t),x_{l}(t))(m_{t_{2n+2}}-\langle x_{i}(t),v\rangle)}\\
		\vspace{0.3cm}\displaystyle{\hspace{2cm}=\frac{1}{N-1}\sum_{l:l\neq i}\psi(x_{i}(t),x_{l}(t))(\langle x_{i}(t),v\rangle)-m_{t_{2n+2}})}\\
		\displaystyle{\hspace{2cm}\leq K(\langle x_{i}(t),v\rangle)-m_{t_{2n+2}}).}
	\end{array}$$
	Therefore, the Gronwall's inequality yields \begin{equation}\label{1}
		\begin{split}
			\langle x_{i}(t),v\rangle&\leq e^{K(t-t_{2n+1})}\langle x_{i}(t_{2n+1}),v\rangle-Km_{t_{2n+2}}\int_{t_{2n+1}}^{t}e^{K(t-s)}ds\\&=e^{K(t-t_{2n+1})}\langle x_{i}(t_{2n+1}),v\rangle+m_{t_{2n+2}}(1-e^{K(t-t_{2n+1})}).
		\end{split}
	\end{equation}
On the other hand, using the second inequality in \eqref{dis}, we get 
$$\begin{array}{l}
	\vspace{0.3cm}\displaystyle{\frac{d}{dt}\langle x_{i}(t),v\rangle=-\frac{1}{N-1}\sum_{l:l\neq j}\psi(x_{j}(t),x_{l}(t))(\langle x_{l}(t),v\rangle-\langle x_{j}(t),v\rangle)}\\
	\vspace{0.3cm}\displaystyle{\hspace{2cm}\geq -\frac{1}{N-1}\sum_{l:l\neq j}\psi(x_{j}(t),x_{l}(t))(M_{t_{2n+2}}-\langle x_{j}(t),v\rangle)}\\
	\vspace{0.3cm}\displaystyle{\hspace{2cm}=\frac{1}{N-1}\sum_{l:l\neq j}\psi(x_{j}(t),x_{l}(t))(\langle x_{j}(t),v\rangle)-M_{t_{2n+2}})}\\
	\displaystyle{\hspace{2cm}\geq K(\langle x_{j}(t),v\rangle)-M_{t_{2n+2}}).}
\end{array}$$
Hence, using the Gronwall's inequality, we can write
\begin{equation}\label{1b}
	\begin{split}
		\langle x_{j}(t),v\rangle&\geq e^{K(t-t_{2n+1})}\langle x_{j}(t_{2n+1}),v\rangle-KM_{t_{2n+2}}\int_{t_{2n+1}}^{t}e^{K(t-s)}ds\\&=e^{K(t-t_{2n+1})}\langle x_{j}(t_{2n+1}),v\rangle+M_{t_{2n+2}}(1-e^{K(t-t_{2n+1})}).
	\end{split}
\end{equation}
Thus, combining \eqref{1} and \eqref{1b}, we can conclude that, for all $t\in [t_{2n+1},t_{2n+2}]$, it holds 
$$\begin{array}{l}
    \vspace{0.3cm}\displaystyle{\langle x_{i}(t)-x_{j}(t),v\rangle=\langle x_{i}(t),v\rangle-\langle x_{j}(t),v\rangle}\\
	\vspace{0.3cm}\displaystyle{\hspace{2.9cm}\leq e^{K(t-t_{2n+1})}\langle x_{i}(t_{2n+1}),v\rangle+m_{t_{2n+2}}(1-e^{K(t-t_{2n+1})})}\\
	\vspace{0.3cm}\displaystyle{\hspace{3.4cm}-e^{K(t-t_{2n+1})}\langle x_{j}(t_{2n+1}),v\rangle-M_{t_{2n+2}}(1-e^{K(t-t_{2n+1})})}\\	\vspace{0.3cm}\displaystyle{\hspace{2.9cm}\leq e^{K(t-t_{2n+1})}\langle x_{i}(t_{2n+1})-x_{j}(t_{2n+1}),v\rangle+(m_{t_{2n+2}}-M_{t_{2n+2}})(1-e^{K(t-t_{2n+1})})}\\	\vspace{0.3cm}\displaystyle{\hspace{2.9cm}=e^{K(t-t_{2n+1})}\langle x_{i}(t_{2n+1})-x_{j}(t_{2n+1}),v\rangle+(M_{t_{2n+2}}-m_{t_{2n+2}})(e^{K(t-t_{2n+1})}-1)}\\	\displaystyle{\hspace{2.9cm}\leq e^{K(t-t_{2n+1})}d(t_{2n+1})+d(t_{2n+2})(e^{K(t-t_{2n+1})}-1).}
\end{array}$$
For $t=t_{2n+2}$, we obtain $$
	d(t_{2n+2})=\langle x_{i}(t_{2n+2})-x_{j}(t_{2n+2}),v\rangle\leq 
	e^{K(t_{2n+2}-t_{2n+1})}d(t_{2n+1})+d(t_{2n+2})(e^{K(t_{2n+2}-t_{2n+1})}-1),
$$
from which
$$d(t_{2n+2})(2-e^{K(t_{2n+2}-t_{2n+1})})\leq e^{K(t_{2n+2}-t_{2n+1})}d(t_{2n+1}).$$
Thus, since $2-e^{K(t_{2n+2}-t_{2n+1})}>0$ from \eqref{tn}, we have that
$$d(t_{2n+2})\leq \frac{e^{K(t_{2n+2}-t_{2n+1})}}{2-e^{K(t_{2n+2}-t_{2n+1})}}d(t_{2n+1}),$$i.e. \eqref{boundgrowth} is proved.
\end{proof}
\section{Proof of Theorem \ref{cons}}\label{consensus_sec}
	
	\setcounter{equation}{0}
	
	\begin{proof}[Proof of Theorem \ref{cons}]
		Let $\{x_{i}\}_{i=1,\dots,N}$ be solution to \eqref{onoff}, \eqref{incond}. Then, for all $n\in \mathbb{N}_{0}$, using \eqref{n-2}, \eqref{cn} and \eqref{boundgrowth}we have that
		$$d(t_{2n+2})\leq \frac{e^{K(t_{2n+2}-t_{2n+1})}}{2-e^{K(t_{2n+2}-t_{2n+1})}}d(t_{2n+1})$$
		$$\leq \frac{e^{K(t_{2n+2}-t_{2n+1})}}{2-e^{K(t_{2n+2}-t_{2n+1})}}\max\left\{1-e^{-K(t_{2n+1}-t_{2n})},1-\frac{\psi_{0}}{K}(1-e^{-K(t_{2n+1}-t_{2n})})\right\}d(t_{2n}).$$
		Thus, using an induction argument, we get
		$$d(t_{2n+2})\leq \prod_{p=0}^{n}\left(\frac{e^{K(t_{2p+2}-t_{2p+1})}}{2-e^{K(t_{2p+2}-t_{2p+1})}}\max\left\{1-e^{-K(t_{2p+1}-t_{2p})},1-\frac{\psi_{0}}{K}(1-e^{-K(t_{2p+1}-t_{2p})})\right\}\right)d(0)$$
		$$=e^{\sum_{p=0}^{\infty}\ln \left(\frac{e^{K(t_{2p+2}-t_{2p+1})}}{2-e^{K(t_{2p+2}-t_{2p+1})}}\max\left\{1-e^{-K(t_{2p+1}-t_{2p})},1-\frac{\psi_{0}}{K}(1-e^{-K(t_{2p+1}-t_{2p})})\right\}\right)}d(0)$$
		$$=e^{\sum_{p=0}^{\infty}\left[\ln \left(\frac{e^{K(t_{2p+2}-t_{2p+1})}}{2-e^{K(t_{2p+2}-t_{2p+1})}}\right)+\ln\left(\max\left\{1-e^{-K(t_{2p+1}-t_{2p})},1-\frac{\psi_{0}}{K}(1-e^{-K(t_{2p+1}-t_{2p})})\right\}\right)\right]}d(0)$$
		Now, $\sum_{p=0}^{\infty}\ln \left(\frac{e^{K(t_{2p+2}-t_{2p+1})}}{2-e^{K(t_{2p+2}-t_{2p+1})}}\right)<+\infty$ from \eqref{sum1}. Then, the solution $\{x_{i}\}_{i=1,\dots,N}$ converges to consensus since the following condition is satisfied from \eqref{sum2}:
		$$\sum_{p=0}^{\infty}\ln\left(\max\left\{1-e^{-K(t_{2p+1}-t_{2p})},1-\frac{\psi_{0}}{K}(1-e^{-K(t_{2p+1}-t_{2p})})\right\}\right)=-\infty.$$
	\end{proof}
Under a stronger condition on the sequence $\{t_{n}\}_{n}$, the following exponential convergence to consensus holds.
\begin{thm}\label{expcons}
	Let $\psi:\RR^d\times \RR^d\rightarrow\RR$ be a positive, bounded, continuous function. Assume that the sequence $\{t_{n}\}_{n}$ of definition \eqref{alpha1} satisfies \eqref{tn}. Assume \eqref{bad} and that the following condition holds:
	\begin{equation}\label{addcond1}
		\sup_{n\in \mathbb{N}}\left(\frac{e^{K(t_{2n+2}-t_{2n+1})}}{2-e^{K(t_{2n+2}-t_{2n+1})}}\max\left\{1-e^{-K(t_{2n+1}-t_{2n})},1-\frac{\psi_{0}}{K}(1-e^{-K(t_{2n+1}-t_{2n})})\right\}\right)=c<1.
	\end{equation}
	 Then, every solution $\{x_{i}\}_{i=1,\dots,N}$ to \eqref{onoff} with the initial conditions \eqref{incond} satisfies the following exponential decay estimate
	 \begin{equation}\label{expest}
	 	d(t)\leq e^{-\gamma\left(t-\frac{\ln2}{K}-T\right) }d(0),\quad \forall t\geq 0,
	 \end{equation}
	 for two suitable positive constants $\gamma$ and $T$, independent of $N$.
\end{thm}
\begin{oss}
	The assumption \eqref{addcond1} implies  \eqref{sum2}. Indeed, if \eqref{addcond1} holds, 
	$$\sum_{p=0}^{\infty}\ln \left(\frac{e^{K(t_{2p+2}-t_{2p+1})}}{2-e^{K(t_{2p+2}-t_{2p+1})}}\max\left\{1-e^{-K(t_{2p+1}-t_{2p})},1-\frac{\psi_{0}}{K}(1-e^{-K(t_{2p+1}-t_{2p})})\right\}\right)$$$$\leq \sum_{p=0}^{\infty}\ln c=-\infty,$$
	i.e. \begin{equation}\label{sumtot}
		\sum_{p=0}^{\infty}\ln \left(\frac{e^{K(t_{2p+2}-t_{2p+1})}}{2-e^{K(t_{2p+2}-t_{2p+1})}}\max\left\{1-e^{-K(t_{2p+1}-t_{2p})},1-\frac{\psi_{0}}{K}(1-e^{-K(t_{2p+1}-t_{2p})})\right\}\right)=-\infty.
	\end{equation}
	Therefore, being $\frac{e^{K(t_{2p+2}-t_{2p+1})}}{2-e^{K(t_{2p+2}-t_{2p+1})}}>1$, it comes that
		$$\sum_{p=0}^{\infty}\ln \left(\max\left\{1-e^{-K(t_{2p+1}-t_{2p})},1-\frac{\psi_{0}}{K}(1-e^{-K(t_{2p+1}-t_{2p})})\right\}\right)$$$$\leq\sum_{p=0}^{\infty}\ln \left(\frac{e^{K(t_{2p+2}-t_{2p+1})}}{2-e^{K(t_{2p+2}-t_{2p+1})}}\max\left\{1-e^{-K(t_{2p+1}-t_{2p})},1-\frac{\psi_{0}}{K}(1-e^{-K(t_{2p+1}-t_{2p})})\right\}\right).$$
		Then, from \eqref{sumtot}, the condition \eqref{sum2} is satisfied. 
	\end{oss}
		\begin{proof}[Proof of Theorem \ref{expcons}]
			Let $\{x_{i}\}_{i=1,\dots,N}$ be solution to \eqref{onoff}, \eqref{incond}. Then, since \eqref{bad} holds, for all $n\in \mathbb{N}_{0}$, 	$$d(t_{2n+2})\leq \frac{e^{K(t_{2n+2}-t_{2n+1})}}{2-e^{K(t_{2n+2}-t_{2n+1})}}\max\left\{1-e^{-K(t_{2n+1}-t_{2n})},1-\frac{\psi_{0}}{K}(1-e^{-K(t_{2n+1}-t_{2n})})\right\}d(t_{2n}).$$
			Therefore, using \eqref{addcond1}, we get \begin{equation}\label{2p}
				d(t_{2n+2})\leq cd(t_{2n}),
			\end{equation}
			with $c\in (0,1)$. As a consequence, using an induction argument,
			\begin{equation}\label{diamin}
				d(t_{2n})\leq c^{n}d(0),\quad \forall n\in \mathbb{N}_{0}.
			\end{equation}
			Now, let $t\geq 0$. Then, there exists $n\in \mathbb{N}_{0}$ such that $t\in [t_{2n},t_{2n+2})$. Thus, if $t\in [t_{2n},t_{2n+1}]$, using \eqref{dist}, 
			$$d(t)\leq d(t_{2n}).$$
			On the other hand, if $t\in (t_{2n+1},t_{2n+2})$, from \eqref{dist-1} and \eqref{2p} we get
			$$d(t)\leq d(t_{2n+2})\leq cd(t_{2n}).$$
			Therefore, being $c<1$, in both cases
			$$d(t)\leq d(t_{2n}).$$
			So, using \eqref{diamin}, we can write
			$$d(t)\leq c^{n}d(0)=e^{-n\ln \left(\frac{1}{c}\right)}d(0).$$
			At this point, we distinguish two different situations. 
			\\Case I) Assume that $$T:=\sup_{n\in \mathbb{N}_0}(t_{2n+1}-t_{2n})<+\infty.$$
			So setting 
			$$\gamma:=\ln \left(\frac{1}{c}\right)\frac{1}{\frac{\ln2}{K}+T},$$
			it comes that $$d(t)\leq e^{-n\gamma\left(\frac{\ln2}{K}+T\right)}$$
			Now, using \eqref{tn}, it holds that $t_{2n+2}\leq (n+1)\left(\frac{\ln2}{K}+T\right)$. Thus, being $t\leq t_{2n+2}$, we can conclude that $$d(t)\leq e^{-\gamma\left(t-\frac{\ln2}{K}-T\right)}d(0),$$
			which proves \eqref{expest}.
			\\Case II) Assume that $$\sup_{n\in \mathbb{N}}(t_{2n+1}-t_{2n})=+\infty.$$
			We pick $\tilde{T}>0$. Without loss of generality, eventually splitting the intervals in which the weight function $\alpha=1$ in subintervals of length at most $\tilde{T}$, we can assume that
			\begin{equation}\label{T}
				t_{2n+1}-t_{2n}\leq \tilde{T},\quad \forall n\in \mathbb{N}_0.
			\end{equation}
		Then, setting $$\gamma:=\ln\left(\frac{1}{c}\right)\frac{1}{\frac{\ln2}{K}+\tilde{T}}.$$
		Then, reasoning as in the previous case, we get that $$d(t)\leq e^{-\gamma\left(t-\frac{\ln2}{K}-\tilde{T}\right)}d(0),$$
		which proves \eqref{expest}.
		\end{proof}
	\section{A particular case}
	In this section, we analyze a Hegselmann-Krause model of the type \eqref{onoff} with communication rates $a_{ij}$ of the form
	\begin{equation}\label{infnotgen}
		a_{ij}(t):=\frac{1}{N-1}\tilde{\psi}(\lvert x_{i}(t)-x_{j}(t)\rvert),
	\end{equation}
where $\tilde{\psi}:\RR\rightarrow\RR$ is a positive, continuous, and bounded function whose supremum norm is given by
$$\tilde{K}:=\lVert\tilde{\psi}\rVert_{\infty}.$$
A class of functions that is included in this analysis is the one of all continuous, nonincreasing functions $\tilde{\psi}$ satisfying
$$\tilde{\psi}(x)\to 0,\quad \text{as }x\to +\infty,$$ 
that is the class of functions more often appearing in the applications.
\\The following theorem guarantees the asymptotic consensus for the solution of the Hegselmann-Krause type model \eqref{onoff} with the initial conditions \eqref{incond} and communication rates given by \eqref{infnotgen}. 
\begin{thm}\label{consnotgen}
	Let $\tilde{\psi}:\RR\rightarrow\RR$ be a positive, bounded, continuous function. Assume that the sequence $\{t_{n}\}_{n}$ of definition \eqref{alpha1} satisfies \begin{equation}\label{tntilde}
		t_{2n+2}-t_{2n+1}<\frac{\ln 2}{\tilde K},\quad \forall n\in \mathbb{N}_0.
	\end{equation}
Moreover, assume that the two following conditions hold:
	\begin{equation}\label{sum1tilde}
		\sum_{p=0}^{\infty}\ln\left(\frac{e^{\tilde K(t_{2p+2}-t_{2p+1})}}{2-e^{\tilde K(t_{2p+2}-t_{2p+1})}}\right)<+\infty,
	\end{equation}
	\begin{equation}\label{sum2nob}
		\sum_{p=0}^{\infty}\ln\left(\max\left\{1-e^{-\tilde K(t_{2p+1}-t_{2p})},1-\frac{\tilde{\psi}_{0}}{\tilde K}(1-e^{-\tilde K(t_{2p+1}-t_{2p})})\right\}\right)=-\infty.
	\end{equation} 
where \begin{equation}\label{tildepsi}
	\tilde{\psi}_{0}:=\min_{\lvert y\rvert\leq M^{0}}\tilde\psi(y),
\end{equation}
being
\begin{equation}\label{boundunif}
	\tilde M^{0}:=e^{\sum_{p=0}^{\infty}\ln \left(\frac{e^{\tilde K(t_{2n+2}-t_{2n+1})}}{2-e^{\tilde K(t_{2n+2}-t_{2n+1})}}\right)}d(0).
\end{equation}
Then, every solution $\{x_{i}\}_{i=1,\dots,N}$ to \eqref{onoff} with the initial conditions \eqref{incond} converges to consensus.
\end{thm}
\begin{proof}
	Using analogous arguments to the ones employed in the more general case examined in the previous sections, all the results from Lemma \ref{L1} to Proposition \ref{L4} of Section 2 are still valid, up to substituting the positive constant $K$ with the constant of boundness of the now considered influence function $\tilde{K}$. Therefore, for all $n\in \mathbb{N}_{0}$, from \eqref{dist-1} and \eqref{boundgrowth} we have that
	$$d(t_{2n+2})\leq \frac{e^{\tilde K(t_{2n+2}-t_{2n+1})}}{2-e^{\tilde K(t_{2n+2}-t_{2n+1})}}d(t_{2n+1})\leq \frac{e^{\tilde K(t_{2n+2}-t_{2n+1})}}{2-e^{\tilde K(t_{2n+2}-t_{2n+1})}}d(t_{2n}).$$
	Then, $$d(t_{2n})\leq \prod_{p=0}^{n}\left(\frac{e^{\tilde K(t_{2n+2}-t_{2n+1})}}{2-e^{\tilde K(t_{2n+2}-t_{2n+1})}}\right)d(0)=e^{\sum_{p=0}^{n}\ln \left(\frac{e^{\tilde K(t_{2n+2}-t_{2n+1})}}{2-e^{\tilde K(t_{2n+2}-t_{2n+1})}}\right)}d(0)$$$$\leq e^{\sum_{p=0}^{\infty}\ln \left(\frac{e^{\tilde K(t_{2n+2}-t_{2n+1})}}{2-e^{\tilde K(t_{2n+2}-t_{2n+1})}}\right)}d(0).$$
	We have so proven that
	\begin{equation}
		d(t_{2n})\leq \tilde M^{0},\quad \forall n \in \mathbb{N}_{0}.
	\end{equation}
	As a consequence, for all $t\geq 0$, since $t\in [t_{2n},t_{2n+2})$, for some $n\in \mathbb{N}_{0}$, using \eqref{dist} and \eqref{dist-1} it comes that 
	$$d(t)\leq \tilde M^{0}.$$
	Thus, for all $i,j=1,\dots,N$
	$$\lvert x_{i}(t)-x_{j}(t)\rvert\leq d(t)\leq \tilde M^{0},$$
	from which
	\begin{equation}
		\psi(\lvert x_{i}(t)-x_{j}(t)\rvert)\geq \tilde{\psi}^{0}.
	\end{equation}
Now, we repeat the proof of Lemma \ref{L5} up to substituting $\psi_{0}$ with the constant $\tilde{\psi}_{0}$. To be precise, we have that
$$d(t_{2n+1})\leq \max\left\{1-e^{-\tilde K(t_{2n+1}-t_{2n})},1-\frac{\tilde{\psi}_{0}}{\tilde K}(1-e^{-\tilde K(t_{2n+1}-t_{2n})})\right\}d(t_{2n}),$$
where $\max\left\{1-e^{-\tilde K(t_{2n+1}-t_{2n})},1-\frac{\tilde{\psi}_{0}}{\tilde K}(1-e^{-\tilde K(t_{2n+1}-t_{2n})})\right\}\in (0,1)$.
\\Then, taking into account of the conditions \eqref{sum1tilde} and \eqref{sum2nob}, the proof follows analogously to the one of Theorem \ref{cons}.
\end{proof}
Again, with an additional condition on the sequence $\{t_{n}\}_{n}$, one can prove the following exponential convergence to consensus result. We omit the proof since it is analogous to the one of Theorem \ref{expcons}.
\begin{thm}
	Let $\psi:\RR\rightarrow\RR$ be a positive, bounded, continuous function. Assume that the sequence $\{t_{n}\}_{n}$ of definition \eqref{alpha1} satisfies \eqref{tntilde}. Assume also that the following conditions hold:
	\begin{equation}\label{badtilde}
		\sum_{p=0}^{+\infty}(t_{2p+2}-t_{2p+1})<+\infty,
	\end{equation}
	\begin{equation}\label{addcond2}
		\sup_{n\in \mathbb{N}}\left(\frac{e^{\tilde K(t_{2n+2}-t_{2n+1})}}{2-e^{\tilde K(t_{2n+2}-t_{2n+1})}}\max\left\{1-e^{-\tilde K(t_{2n+1}-t_{2n})},1-\frac{\tilde{\psi}_{0}}{\tilde K}(1-e^{-\tilde K(t_{2n+1}-t_{2n})})\right\}\right)=c<1,
	\end{equation}
where $\tilde{\psi}_0$ is the positive constant in \eqref{tildepsi}. Then, every solution $\{x_{i}\}_{i=1,\dots,N}$ to \eqref{onoff} with the initial conditions \eqref{incond} satisfies the following exponential decay estimate
	\begin{equation}\label{expestnb}
		d(t)\leq e^{-\gamma\left(t-\frac{\ln2}{\tilde K}-T\right) }d(0),\quad \forall t\geq 0,
	\end{equation}
	for two suitable positive constants $\gamma$, independent of $N$, and $T$.
\end{thm}
\section{The Cucker-Smale model with attractive-repulsive interaction}
Consider a finite set of $N\in\N$ particles, with $N\geq 2 $. Let $x_{i}(t),v_{i}(t)\in \RR^d$ be the opinion and the velocity of the $i$-th particle at time $t$, respectively. The interactions between the elements of the system are described by the following  Cucker-Smale model with attractive-repulsive interaction
\begin{equation}\label{onoffcs}
	\begin{cases}
		\frac{d}{dt}x_{i}(t)=v_{i}(t),&\quad t>0,	\,\,\forall i=1,\dots,N,\\\frac{d}{dt}v_{i}(t)=\underset{j:j\neq i}{\sum}\alpha(t) a_{ij}(t)(x_{j}(t)-x_{i}(t)),&\quad t>0,	\,\,\forall i=1,\dots,N,
	\end{cases}
\end{equation}
with initial conditions
\begin{equation}\label{incondcs}
	\begin{cases}
		x_{i}(0)=x^{0}_{i}\in \mathbb{R}^{d},&\quad \forall i=1,\dots,N,\\v_{i}(0)=v^{0}_{i}\in \mathbb{R}^{d},&\quad \forall i=1,\dots,N.
	\end{cases}		
\end{equation}
Here, the communication rates $a_{ij}$ are of the form \eqref{infnotgen}, being $\tilde\psi:\RR\rightarrow \RR$ a continuous, positive, bounded influence function with supremum norm
$$\tilde{K}=\lVert\tilde\psi\rVert_{\infty}. $$The weight function $\alpha:[0,+\infty)\rightarrow\{-1,1\}$ is defined as in \eqref{alpha1}. Moreover, we assume that the sequence $\{t_n\}_n$ in the definition \eqref{alpha1} satisfies the following property:
\begin{equation}\label{tnbuoni}
	t_{2n+1}-t_{2n}>\frac{1}{\tilde K},\quad \forall n\in \mathbb{N}_0.
\end{equation}
Also, we set
\begin{equation}\label{tildeM0}
	\tilde{M}_{n}^{0}:=\max_{i=1,\dots,N}\lvert v_{i}(t_{n})\rvert.
\end{equation}
For each $t\geq 0$, we define the position diameter $d_{X}(\cdot)$ as
$$d_{X}(t):=\max_{i,j=1,\dots,N}\lvert x_{i}(t)-x_{j}(t)\rvert,$$
and the velocity diameter $d_{V}(\cdot)$ as $$d_{V}(t):=\max_{i,j=1,\dots,N}\lvert v_{i}(t)-v_{j}(t)\rvert.$$
\begin{defn} (Unconditional flocking)\label{unflock} We say that a solution $\{(x_{i},v_{i})\}_{i=1,\dots,N}$ to system \eqref{onoffcs} exhibits \textit{asymptotic flocking} if it satisfies the two following conditions:
	\begin{enumerate}
		\item there exists a positive constant $d^{*}$ such that$$\sup_{t\geq0}d_{X}(t)\leq d^{*};$$
		\item$\underset{t \to\infty}{\lim}d_{V}(t)=0.$
	\end{enumerate}
\end{defn}
We will prove the following flocking result.
\begin{thm}\label{af}
	Let $\tilde\psi:\RR^d\times \RR^d\rightarrow\RR$ be a positive, bounded, continuous function satisfying \begin{equation}\label{infint}
		\int_{0}^{\infty}\min_{r\in [0,x]}\tilde \psi(r)dx=+\infty.
	\end{equation} Assume that the sequence $\{t_{n}\}_{n}$ of definition \eqref{alpha1} satisfies \eqref{tntilde}, \eqref{tnbuoni}. Moreover, assume that the condition \eqref{sum1tilde} holds.
	Then, every solution $\{x_{i},v_{i}\}_{i=1,\dots,N}$ to \eqref{onoffcs} with the initial conditions \eqref{incondcs} exhibits asymptotic flocking.
	\end{thm}
\section{The asymptotic flocking result}
\setcounter{equation}{0}

\noindent Let $\{(x_{i},v_{i})\}_{i=1,\dots,N}$ be solution to \eqref{onoffcs} under the initial conditions \eqref{incondcs}. We assume that the hypotheses of Theorem \ref{af} are satisfied. We first present some auxiliary lemmas that will be needed for the proof of Theorem \ref{af}. Since these auxiliary results are analogous to the ones given in Section 3, we omit their proof.
\begin{lem}\label{L1cs}
	For each $v\in \RR^{d}$ and $n\in \mathbb{N}_0,$  we have that 
	\begin{equation}\label{scalprcs}
		\min_{j=1,\dots,N}\langle v_{j}(t_{2n}),v\rangle\leq \langle v_{i}(t),v\rangle\leq \max_{j=1,\dots,N}\langle v_{j}(t_{2n}),v\rangle,
	\end{equation}for all  $t\in [t_{2n},t_{2n+1}]$ and $i=1,\dots,N$. 
\end{lem}

\begin{lem}\label{L2cs}
	For each $n\in \mathbb{N}_{0}$ and $i,j=1,\dots,N$, we get \begin{equation}
		\label{distcs}
		\lvert v_{i}(s)-v_{j}(t)\rvert\leq d_{V}(t_{2n}), \quad\forall s,t\in [t_{2n},t_{2n+1}].
	\end{equation} 
\end{lem}
\begin{oss}
	Let us note that from \eqref{distcs}, in particular, it follows that
	\begin{equation}\label{deccs}
		d_{V}(t_{2n+1})\leq d_{V}(t_{2n}),\quad \forall n\in \mathbb{N}_{0}.
	\end{equation}
	
\end{oss} 
\begin{lem}\label{L1-1cs}
	For each $v\in \RR^{d}$ and $n\in \mathbb{N}_0,$  we have that 
	\begin{equation}\label{scalpr-1cs}
		\min_{j=1,\dots,N}\langle v_{j}(t_{2n+2}),v\rangle\leq \langle v_{i}(t),v\rangle\leq \max_{j=1,\dots,N}\langle v_{j}(t_{2n+2}),v\rangle,
	\end{equation}for all  $t\in [t_{2n+1},t_{2n+2}]$ and $i=1,\dots,N$. 
\end{lem}

\begin{lem}\label{L2-1cs}
	For each $n\in \mathbb{N}_{0}$ and $i,j=1,\dots,N$, we get \begin{equation}
		\label{dist-1cs}
		\lvert v_{i}(s)-v_{j}(t)\rvert\leq d_{V}(t_{2n+2}), \quad\forall s,t\in [t_{2n+1},t_{2n+2}].
	\end{equation} 
\end{lem}

\begin{oss}
	Let us note that from \eqref{distcs}, in particular, it follows that
	\begin{equation}\label{crecs}
		d_{V}(t_{2n+2})\geq d_{V}(t_{2n+1}),\quad \forall n\in \mathbb{N}_{0}.
	\end{equation}
	
\end{oss}
\begin{lem}\label{L3cs}
	For every $i=1,\dots,N,$ we have that \begin{equation}\label{boundsolcs}
		\lvert v_{i}(t)\rvert\leq \tilde M_{2n}^{0},\quad \forall t\in [t_{2n},t_{2n+1}],
	\end{equation}
	where $\tilde {M}_{2n}^{0}$	is the positive constant in \eqref{tildeM0}.
\end{lem}
\begin{lem}
	For all $t\geq 0$, we have that
	\begin{equation}\label{unifboundsolcs}
		\max_{i=1,\dots,N}\lvert v_{i}(t)\rvert\leq \tilde M^{0},
	\end{equation}
	where 
	\begin{equation}\label{mocs}
		\tilde M^0:=e^{\tilde K\sum_{p=0}^{\infty}(t_{2p+2}-t_{2p+1})}\tilde{M}^{0}_{0}
	\end{equation}
\end{lem}

\begin{prop}\label{L4cs}
	For all $i,j=1,\dots,N$,  unit vector $v\in \RR^{d}$ and $n\in\mathbb{N}_{0}$ we have that 
	\begin{equation}\label{4cs}
		\langle v_{i}(t)-v_{j}(t),v\rangle\leq e^{-\tilde K(t-\bar{t})}\langle v_{i}(\bar{t})-x_{j}(\bar{t}),v\rangle+(1-e^{-\tilde K(t-\bar{t})})d_{V}(t_{2n}),
	\end{equation}
	for all $t_{2n+1}>t\geq \bar{t}\geq t_{2n}$. 
\end{prop}

\begin{lem}
	For all $n\in \mathbb{N}_{0}$, we have that
	\begin{equation}\label{boundgrowthcs}
		d_{V}(t_{2n+2})\leq \frac{e^{\tilde K(t_{2n+2}-t_{2n+1})}}{2-e^{\tilde K(t_{2n+2}-t_{2n+1})}}d_{V}(t_{2n+1}).
	\end{equation}
\end{lem} 
\begin{prop}
	Assume \eqref{sum1tilde}. Then, for all $t\geq 0$, it holds that
	\begin{equation}\label{veldiambound}
		d_{V}(t)\leq \bar{M}^{0}, 
	\end{equation}
where
\begin{equation}\label{boundunifcs}
	\bar{M}^{0}:=e^{\sum_{p=0}^{\infty}\ln \left(\frac{e^{\tilde K(t_{2n+2}-t_{2n+1})}}{2-e^{\tilde K(t_{2n+2}-t_{2n+1})}}\right)}d_{V}(0).
\end{equation}
\end{prop}	
\begin{proof}
For all $n\in \mathbb{N}_{0}$, from \eqref{deccs} and \eqref{boundgrowthcs} we have that
	$$d_{V}(t_{2n+2})\leq \frac{e^{\tilde K(t_{2n+2}-t_{2n+1})}}{2-e^{\tilde K(t_{2n+2}-t_{2n+1})}}d_{V}(t_{2n+1})\leq \frac{e^{\tilde K(t_{2n+2}-t_{2n+1})}}{2-e^{\tilde K(t_{2n+2}-t_{2n+1})}}d_{V}(t_{2n}).$$
	Then, $$d_{V}(t_{2n})\leq \prod_{p=0}^{n}\left(\frac{e^{\tilde K(t_{2n+2}-t_{2n+1})}}{2-e^{\tilde K(t_{2n+2}-t_{2n+1})}}\right)d_{V}(0)=e^{\sum_{p=0}^{n}\ln \left(\frac{e^{\tilde K(t_{2n+2}-t_{2n+1})}}{2-e^{\tilde K(t_{2n+2}-t_{2n+1})}}\right)}d_{V}(0)$$$$\leq e^{\sum_{p=0}^{\infty}\ln \left(\frac{e^{\tilde K(t_{2n+2}-t_{2n+1})}}{2-e^{\tilde K(t_{2n+2}-t_{2n+1})}}\right)}d_{V}(0).$$
	We have so proven that
	\begin{equation}
		d_{V}(t_{2n})\leq \bar{M}^{0},\quad \forall n \in \mathbb{N}_{0}.
	\end{equation}
	As a consequence, for all $t\geq 0$, since $t\in [t_{2n},t_{2n+2})$, for some $n\in \mathbb{N}_{0}$, using \eqref{dist} and \eqref{dist-1} we can conclude that \eqref{veldiambound} holds true.
	\end{proof}

Now, we pick $T> \frac{\ln 2}{\tilde K}$. We can assume, eventually splitting the intervals of positive interaction into subintervals of length at most $T$, that
\begin{equation}\label{boundbuoni}
	t_{2n+1}-t_{2n}\leq T,\quad \forall n\in \mathbb{N}_0.
\end{equation}
As a consequence, from \eqref{tntilde} and \eqref{boundbuoni}, being $T> \frac{\ln 2}{\tilde K}$, we can write 
\begin{equation}\label{boundbc}
	t_{n+1}-t_{n}\leq T,\quad \forall n\in \mathbb{N}_0.
\end{equation} 

\begin{defn}
	For all $t\geq 0$, we define
	\begin{equation}\label{psit}
		\tilde\psi_{t}:=\min\left\{\tilde\psi(r):r\in \left[0,\max_{s\in [0,t]}d_X(s)\right]\right\},
	\end{equation}
	and
	\begin{equation}\label{defphi}
		\phi(t):=\min\left\{e^{-\tilde KT}\tilde\psi_{t},\frac{e^{-\tilde KT}}{T}\right\}.
	\end{equation}
\end{defn}
\begin{oss}
	Let us note that, for all $t\geq 0$ and $i,j=1,\dots,N$, 
	$$\lvert x_{i}(t)-x_{j}(t)\rvert\leq \max_{s\in [0,t]}d_X(s).$$
	As a consequence, it holds that 
	\begin{equation}\label{infbou}
		\tilde\psi (\lvert x_{i}(t)-x_{j}(t)\rvert)\geq \tilde\psi_{t}>0, \quad \forall t\geq 0,\,i,j=1,\dots,N.
	\end{equation}
\end{oss}

\begin{prop}\label{L5cs}
	For all $n\in \mathbb{N}_{0}$, 
	\begin{equation}\label{n-2cs}
		d_{V}(t_{2n+1})\leq \left(1-\int_{t_{2n}}^{t_{2n+1}}\phi(s)ds\right)d_{V}(t_{2n}).
	\end{equation}
\end{prop}
\begin{oss}
	Let us note that $$\int_{t_{2n}}^{t_{2n+1}}\phi(s)ds\in (0,1),\quad \forall n\in \mathbb{N}_0,$$
	so that $$1-\int_{t_{2n}}^{t_{2n+1}}\phi(s)ds\in (0,1).$$
\end{oss}
\begin{proof}[Proof of Proposition \ref{L5cs}]
	Let $n\in \mathbb{N}_{0}$. Trivially, if $d_V(t_{2n+1})=0$, then of course inequality \eqref{n-2cs} holds. So, suppose $d_V(t_{2n+1})>0$. Let $i,j=1,\dots,N$ be such that $d_V(t_{2n+1})=\lvert v_{i}(t_{2n+1})-v_{j}(t_{2n+1})\rvert$. We set $$v=\frac{v_{i}(t_{2n+1})-v_{j}(t_{2n+1})}{\lvert v_{i}(t_{2n+1})-v_{j}(t_{2n+1})\rvert}.$$
	Then, $v$ is a unit vector for which we can write $$d_V(t_{2n+1})=\langle v_{i}(t_{2n+1})-v_{j}(t_{2n+1}),v\rangle.$$
	Let us define $$M_{t_{2n}}=\max_{l=1,\dots,N} v_{l}(t_{2n}),v\rangle,$$
	$$m_{t_{2n}}=\min_{l=1,\dots,N}\langle v_{l}(t_{2n}),v\rangle.$$
	Then $M_{t_{2n}}-m_{t_{2n}}\leq d_V(t_{2n})$. 
	\\Now, we distinguish two different situations.
	\par\textit{Case I.} Assume that there exists $\bar{t}\in [t_{2n},t_{2n+1}]$ such that 
	$$\langle v_{i}(\bar{t})-v_{j}(\bar{t}),v\rangle<0.$$
	Then, from \eqref{4cs} with $t_{2n+1}\geq \bar{t}\geq t_{2n}$, we have \begin{equation}\label{t0cs}
		\begin{split}
			d_V(t_{2n+1})&\leq e^{-\tilde K(t_{2n+1}-\bar{t})}\langle v_{i}(\bar{t})-v_{j}(\bar{t}),v\rangle+(1-e^{-\tilde K(t_{2n+1}-\bar{t})})d_V(t_{2n})\\&\leq (1-e^{-\tilde K(t_{2n+1}-\bar{t})})d_V(t_{2n})\\&\leq (1-e^{-\tilde KT})d_V(t_{2n})\\&\leq \left(1-\int_{t_{2n}}^{t_{2n+1}}\phi(s)ds\right)d_V(t_{2n}).
		\end{split}
	\end{equation}
	\par\textit{Case II.} Assume it rather holds \begin{equation}\label{poscs}
		\langle v_{i}(t)-v_{j}(t),v\rangle\geq 0,\quad \forall t\in [t_{2n},t_{2n+1}].
	\end{equation}
	Then, for every  $t\in [t_{2n},t_{2n+1}]$, we have that 
	$$
	\begin{array}{l}
		\displaystyle{
			\frac{d}{dt}\langle v_{i}(t)-v_{j}(t),v\rangle=\frac{1}{N-1}\sum_{l:l\neq i}\tilde \psi(\lvert x_{i}(t)-x_{j}(t)\rvert)\langle v_{l}(t)-v_{i}(t),v\rangle}\\
		\displaystyle{\hspace{1.1cm}-\frac{1}{N-1}\sum_{l:l\neq j}\tilde\psi(\lvert x_{i}(t)-x_{j}(t)\rvert)\langle v_{l}(t)-v_{j}(t),v\rangle}\\
		\displaystyle{\hspace{0.6cm}=\frac{1}{N-1}\sum_{l:l\neq i}\tilde\psi(\lvert x_{i}(t)-x_{j}(t)\rvert)(\langle v_{l}(t),v\rangle-M_{t_{2n}}+M_{t_{2n}}-\langle v_{i}(t),v\rangle)}\\
		\displaystyle{\hspace{1.1cm}+\frac{1}{N-1}\sum_{l:l\neq j}\tilde\psi(\lvert x_{i}(t)-x_{j}(t)\rvert)(\langle v_{j}(t),v\rangle-m_{t_{2n}}+m_{t_{2n}}-\langle v_{l}(t),v\rangle)}\\
		\displaystyle{\hspace{5.5 cm}:=S_1+S_2.}
	\end{array}
	$$
	Now, being $t\in [t_{2n},t_{2n+1}]$,  from \eqref{scalprcs} we have that
	\begin{equation}\label{7cs}
		m_{t_{2n}}\leq\langle v_{k}(t),v\rangle\leq M_{t_{2n}}, \quad\forall k=1,\dots,N.
	\end{equation}
	Therefore, we get
	$$
	\begin{array}{l}
		\displaystyle{
			S_1=\frac{1}{N-1}\sum_{l:l\neq i}\tilde\psi(\lvert x_{i}(t)-x_{j}(t)\rvert)(\langle v_{l}(t),v\rangle-M_{t_{2n}})}\\
		\displaystyle{\hspace{0.7 cm}+\frac{1}{N-1}\sum_{l:l\neq i}\tilde\psi(\lvert x_{i}(t)-x_{j}(t)\rvert)(M_{t_{2n}}-\langle v_{i}(t),v\rangle)}\\
		\displaystyle{\hspace{0.4 cm}\leq \frac{1}{N-1}\tilde\psi_{t}\sum_{l:l\neq i}(\langle v_{l}(t),v\rangle-M_{t_{2n}})+\tilde K(M_{t_{2n}}-\langle v_{i}(t),v\rangle)},
	\end{array}
	$$
	and	
	$$
	\begin{array}{l}
		\displaystyle{
			S_2=\frac{1}{N-1}\sum_{l:l\neq j}\tilde\psi(\lvert x_{i}(t)-x_{j}(t)\rvert)(\langle v_{j}(t),v\rangle-m_{t_{2n}})}\\
		\displaystyle{\hspace{0.7 cm}+\frac{1}{N-1}\sum_{l:l\neq j}\tilde\psi(\lvert x_{i}(t)-x_{j}(t)\rvert)(m_{t_{2n}}-\langle v_{l}(t),v\rangle)}\\
		\displaystyle{\hspace{0.4 cm}\leq \tilde K(\langle v_{j}(t),v\rangle-m_{t_{2n}})+\frac{1}{N-1}\tilde\psi_{t}\sum_{l:l\neq j}(m_{t_{2n}}-\langle v_{l}(t),v\rangle).}
	\end{array}
	$$
	Combining this last fact with \eqref{7cs} it comes that 
	$$\begin{array}{l}
		\vspace{0.2cm}\displaystyle{\frac{d}{dt}\langle v_{i}(t)-v_{j}(t),v\rangle\leq \tilde K(M_{t_{2n}}-m_{t_{2n}}-\langle v_{i}(t)-v_{j}(t),v\rangle)}\\
		\vspace{0.2cm}\displaystyle{\hspace{1.8 cm}+\frac{1}{N-1}\tilde\psi_{t}\sum_{l:l\neq i,j}(\langle v_{l}(t),v\rangle-M_{t_{2n}}+m_{t_{2n}}-\langle v_{l}(t),v\rangle)}\\
		\vspace{0.2cm}\displaystyle{\hspace{1.8 cm}+\frac{1}{N-1}\tilde\psi_{t}(\langle v_{j}(t),v\rangle-M_{t_{2n}}+m_{t_{2n}}-\langle v_{i}(t),v\rangle)}\\
		\vspace{0.2cm}\displaystyle{\hspace{1.5 cm}=\tilde \tilde K(M_{t_{2n}}-m_{t_{2n}})-\tilde K\langle v_{i}(t)-v_{j}(t),v\rangle+\frac{N-2}{N-1}\tilde\psi_{t}(-M_{t_{2n}}+m_{t_{2n}})}\\
		\vspace{0.2cm}\displaystyle{\hspace{1.8 cm}+\frac{1}{N-1}\tilde\psi_{t}(\langle v_{j}(t),v\rangle-M_{t_{2n}}+m_{t_{2n}}-\langle v_{i}(t_{2n}),v\rangle)}.
	\end{array}
	$$
	Now, from \eqref{poscs} we get
	$$\begin{array}{l}
		\vspace{0.2cm}\displaystyle{
			\frac{d}{dt}\langle v_{i}(t)-v_{j}(t),v\rangle\leq \tilde K(M_{t_{2n}}-m_{t_{2n}})-\tilde K\langle v_{i}(t)-v_{j}(t),v\rangle
		}\\
		\vspace{0.4 cm}\displaystyle{\hspace{1.8 cm}
			+\frac{N-2}{N-1}\tilde\psi_{t}(-M_{t_{2n}}+m_{t_{2n}})+\frac{1}{N-1}\tilde\psi_{t}(-M_{t_{2n}}+m_{t_{2n}})
		}\\
		\vspace{0.4 cm}\displaystyle{\hspace{1.8 cm}
			-\frac{1}{N-1}\tilde\psi_{t}\langle v_{i}(t)-v_{j}(t),v\rangle
		}\\
		\vspace{0.3 cm}\displaystyle{\hspace{1.5 cm}\leq \tilde K(M_{t_{2n}}-m_{t_{2n}})-\tilde K\langle v_{i}(t)-v_{j}(t),v\rangle+\tilde\psi_{t}(-M_{t_{2n}}+m_{t_{2n}})
		}\\
		\displaystyle{\hspace{1.5 cm}=\left(\tilde K-\tilde\psi_{t}\right)(M_{t_{2n}}-m_{t_{2n}})-\tilde K\langle v_{i}(t)-v_{j}(t),v\rangle.}
	\end{array}$$
	Hence, from Gronwall's inequality it comes that $$\langle v_{i}(t)-v_{j}(t),v\rangle \leq e^{-\tilde K(t-t_{2n})}\langle v_{i}(t_{2n})-v_{j}(t_{2n}),v\rangle$$$$+(M_{t_{2n}}-m_{t_{2n}})\int_{t_{2n}}^{t}\left(\tilde K-\tilde\psi_{s}\right)e^{-\tilde K(t-s)}ds,$$
	for all $t\in [t_{2n},t_{2n+1}]$. In particular, for $t= t_{2n+1}$, from \eqref{distcs} it comes that 
	$$\begin{array}{l}
		\vspace{0.2cm}\displaystyle{d_V(t_{2n+1})\leq e^{-\tilde K(t_{2n+1}-t_{2n})}\langle v_{i}(t_{2n})-v_{j}(t_{2n}),v\rangle+(M_{t_{2n}}-m_{t_{2n}})\int_{t_{2n}}^{t_{2n+1}}(\tilde K-\tilde \psi_{s})e^{-\tilde K(t_{2n+1}-s)}ds}\\
		\vspace{0.2cm}\displaystyle{\hspace{1.3cm}\leq e^{-\tilde K(t_{2n+1}-t_{2n})}\lvert v_{i}(t_{2n})-v_{j}(t_{2n})\rvert +(M_{t_{2n}}-m_{t_{2n}})\int_{t_{2n}}^{t_{2n+1}}(\tilde K-\tilde \psi_{s})e^{-\tilde K(t_{2n+1}-s)}ds}\\
		\vspace{0.2cm}\displaystyle{\hspace{1.3cm}\leq \left(e^{-\tilde K(t_{2n+1}-t_{2n})} +\tilde K\int_{t_{2n}}^{t_{2n+1}}e^{-\tilde K(t_{2n+1}-s)}ds-\int_{t_{2n}}^{t_{2n+1}}\tilde\psi_{s}e^{-\tilde K(t_{2n+1}-s)}ds\right)d_V(t_{2n})}\\
		\vspace{0.2cm}\displaystyle{\hspace{1.3cm}= \left(e^{-\tilde K(t_{2n+1}-t_{2n})} +1-e^{-\tilde K(t_{2n+1}-t_{2n})}-\int_{t_{2n}}^{t_{2n+1}}\tilde\psi_{s}e^{-\tilde K(t_{2n+1}-s)}ds\right)d_V(t_{2n})}\\
		\vspace{0.2cm}\displaystyle{\hspace{1.3cm}=\left(1-\int_{t_{2n}}^{t_{2n+1}}\tilde \psi_{s}e^{-\tilde K(t_{2n+1}-s)}ds\right)d_V(t_{2n})}\\
		\vspace{0.2cm}\displaystyle{\hspace{1.3cm}\leq \left(1-e^{-\tilde KT}\int_{t_{2n}}^{t_{2n+1}}\tilde \psi_{s}ds\right)d_V(t_{2n})}\\
		\vspace{0.2cm}\displaystyle{\hspace{1.3cm}\leq\left(1-\int_{t_{2n}}^{t_{2n+1}}\phi(s)ds\right)d_V(t_{2n}).}
	\end{array}$$
	So, taking into account \eqref{t0cs}, we can conclude that \eqref{n-2cs} holds.
\end{proof}

\begin{proof}[Proof of Theorem \ref{af}]
	Let $\{(x_{i},v_{i})\}_{i=1,\dots,N}$  be solution to \eqref{onoffcs} under the initial conditions \eqref{incondcs}. We define the function $\mathcal{D}:[0,\infty)\rightarrow [0,\infty),$ 
	$$\mathcal{D}(t):=\begin{cases}
		d_{V}(0), &t=0,\\\left(1-\int_{t_{2n}}^{t}\phi(s)ds\right)d_V(t_{2n}), &t\in (t_{2n},t_{2n+1}],\,n\in \mathbb{N}_0,\\\left(1-\int_{t_{2n+1}}^{t}\phi(s)ds\right)d_V(t_{2n+2}),&t\in (t_{2n+1},t_{2n+2}],\,n\in \mathbb{N}_0.
	\end{cases}$$
	By construction, $\mathcal{D}$ is piecewise continuous. Indeed, $\mathcal{D}$ is continuous everywhere except at points $t_n$, $n\in\mathbb{N}_0$. Moreover, for all $n\in \mathbb{N}$, $n\ge 1,$
	\begin{equation}\label{limt2n}
		\lim_{t\to t^{+}_{2n}}\mathcal{D}(t)=d_V(t_{2n})\geq \left(1-\int_{t_{2n-1}}^{t_{2n}}\phi(s)ds\right)d_V(t_{2n})=\mathcal{D}(t_{2n}),
	\end{equation}
	\begin{equation}\label{limt2n+1}
		\begin{array}{l}
			\displaystyle{\lim_{t\to t^{+}_{2n+1}}\mathcal{D}(t)=d_V(t_{2n+2})\leq \frac{e^{\tilde K(t_{2n+2}-t_{2n+1})}}{2-e^{\tilde K(t_{2n+2}-t_{2n+1})}}d_V(t_{2n+1})}\\
\displaystyle{\hspace{1,8 cm}\leq \frac{e^{\tilde K(t_{2n+2}-t_{2n+1})}}{2-e^{\tilde K(t_{2n+2}-t_{2n+1})}}\left(1-\int_{t_{2n}}^{t_{2n+1}}\phi(s)ds\right)d_V(t_{2n})}\\
\displaystyle{
\hspace{2,5 cm}
=\frac{e^{\tilde K(t_{2n+2}-t_{2n+1})}}{2-e^{\tilde K(t_{2n+2}-t_{2n+1})}}\mathcal{D}(t_{2n+1}).}
		\end{array}
	\end{equation}
	Also, $\mathcal{D}$ is nonincreasing in all intervals of the form $(t_{n},t_{n+1}]$, $n\in \mathbb{N}_0$.
	\\Now, notice that, for almost all times $t,$ 
$$\frac{d}{dt}\max_{s\in [0,t]}d_{X}(s)\leq \left\lvert \frac{d}{dt}d_{X}(t)\right\rvert,$$
	since $\max_{s\in [0,t]}d_{X}(s)$ is constant or increases like $d_{X}(t)$. Moreover, for almost all times $$\left\lvert \frac{d}{dt}d_{X}(t)\right\rvert\leq d_{V}(t).$$
	Therefore, for almost all times
	\begin{equation}\label{maxcs}
		\frac{d}{dt}\max_{s\in [0,t]}d_{X}(s)\leq \left\lvert \frac{d}{dt}d_{X}(t)\right\rvert\leq  d_{V}(t).
	\end{equation}
	Next, we define the function $\mathcal{L}:[0,\infty)\rightarrow [0,\infty)$ as follows: $$\mathcal{L}(t):=\mathcal{D}(t)+\int_{0}^{\underset{s\in [0,t]}{\max}d_{X}(s)}\min\left\{e^{- \tilde KT}\min_{\sigma\in [0,r]}\tilde \psi(\sigma),\frac{e^{-\tilde KT}}{T}\right\}dr,$$
	for all $t\geq0$. By definition, $\mathcal{L}$ is piecewise continuous, i.e. $\mathcal{L}$ is continuous everywhere except at points $t_n$, $n\in \mathbb{N}_0$. \\In addition, for each $n\in \mathbb{N}$ and for all $t\in(t_{2n},t_{2n+1}) $, we have that $$\frac{d}{dt}\mathcal{L}(t)=\frac{d}{dt}\mathcal{D}(t)+\min\left\{e^{-\tilde KT}\tilde \psi_{t},\frac{e^{-\tilde KT}}{T}\right\}\frac{d}{dt}\underset{s\in [0,t]}{\max}d_{X}(s)$$$$=\frac{d}{dt}\mathcal{D}(t)+\phi(t)\frac{d}{dt}\underset{s\in [0,t]}{\max}d_{X}(s),$$
	and from \eqref{maxcs} we get $$\begin{array}{l}
		\vspace{0.3cm}\displaystyle{\frac{d}{dt}\mathcal{L}(t)\leq \frac{d}{dt}\mathcal{D}(t)+\phi(t)d_{V}(t)}\\
		\vspace{0.3cm}\displaystyle{\hspace{1.3cm}=-\phi(t)d_V(t_{2n})+\phi(t)d_{V}(t)}\\
		\displaystyle{\hspace{1.3cm}=\phi(t)(d_V(t)-d_V(t_{2n})).}
	\end{array}$$
	Thus, since $d_V(t)\leq d_V(t_{2n})$ from \eqref{distcs}, we can deduce that
	$$\frac{d}{dt}\mathcal{L}(t)\leq 0,\quad \forall t\in (t_{2n},t_{2n+1}).$$
	As a consequence, for all $t_{2n}<s<t\leq t_{2n+1}$, it comes that
	$$\mathcal{L}(t)\leq \mathcal{L}(s).$$ 
	Letting $s\to t_{2n}^+$, we get
	$$\mathcal{L}(t)\leq \lim_{s\to t_{2n}^+}\mathcal{L}(s)=\lim_{s\to t_{2n}^+}\mathcal{D}(s)+\int_{0}^{\underset{s\in [0,t_{2n}]}{\max}d_{X}(s)}\min\left\{e^{-\tilde KT}\min_{\sigma\in [0,r]}\tilde \psi(\sigma),\frac{e^{-\tilde KT}}{T}\right\}dr,$$
	for all $t\in (t_{2n},t_{2n+1}]$. Thus, using \eqref{tn}, \eqref{distcs} and \eqref{limt2n}, we can write
	$$\begin{array}{l}
		\vspace{0.3cm}\displaystyle{\mathcal{L}(t)\leq d_V(t_{2n})+\int_{0}^{\underset{s\in [0,t_{2n}]}{\max}d_{X}(s)}\min\left\{e^{-\tilde KT}\min_{\sigma\in [0,r]}\tilde\psi(\sigma),\frac{e^{-\tilde KT}}{T}\right\}dr}\\
		\vspace{0.3cm}\displaystyle{\hspace{0.5cm}=\frac{1}{\left(1-\int_{t_{2n-1}}^{t_{2n}}\phi(s)ds\right)}\mathcal{D}(t_{2n})+\int_{0}^{\underset{s\in [0,t_{2n}]}{\max}d_{X}(s)}\min\left\{e^{-\tilde KT}\min_{\sigma\in [0,r]}\tilde\psi(\sigma),\frac{e^{-\tilde KT}}{T}\right\}dr}\\
		\vspace{0.3cm}\displaystyle{\hspace{0.5cm}\leq \frac{1}{\left(1-\int_{t_{2n-1}}^{t_{2n}}\phi(s)ds\right)}\left[\mathcal{D}(t_{2n})+\int_{0}^{\underset{s\in [0,t_{2n}]}{\max}d_{X}(s)}\min\left\{e^{-\tilde KT}\min_{\sigma\in [0,r]}\tilde\psi(\sigma),\frac{e^{-\tilde KT}}{T}\right\}dr\right]}\\
		\displaystyle{\hspace{0.5cm}\leq \frac{1}{\left(1-\int_{t_{2n-1}}^{t_{2n}}\phi(s)ds\right)}\mathcal{L}(t_{2n}),}
	\end{array}$$
	for all $t\in (t_{2n},t_{2n+1}]$. So,
	\begin{equation}\label{lia1}
		\mathcal{L}(t)\leq \frac{1}{\left(1-\int_{t_{2n-1}}^{t_{2n}}\phi(s)ds\right)}\mathcal{L}(t_{2n}),\quad \forall t\in [t_{2n},t_{2n+1}].
	\end{equation}
	On the other hand, for all $t\in (t_{2n+1},t_{2n+2})$, we have that 
	$$\frac{d}{dt}\mathcal{L}(t)=\frac{d}{dt}\mathcal{D}(t)+\min\left\{e^{-\tilde KT}\tilde\psi_{t},\frac{e^{-\tilde KT}}{T}\right\}\frac{d}{dt}\underset{s\in [0,t]}{\max}d_{X}(s)$$$$=\frac{d}{dt}\mathcal{D}(t)+\phi(t)\frac{d}{dt}\underset{s\in [0,t]}{\max}d_{X}(s),$$	
	and from \eqref{maxcs} we get $$\begin{array}{l}
		\vspace{0.3cm}\displaystyle{\frac{d}{dt}\mathcal{L}(t)\leq \frac{d}{dt}\mathcal{D}(t)+\phi(t)d_{V}(t)}\\
		\vspace{0.3cm}\displaystyle{\hspace{1.3cm}=-\phi(t)d_V(t_{2n+2})+\phi(t)d_{V}(t)}\\
		\displaystyle{\hspace{1.3cm}=\phi(t)(d_V(t)-d_V(t_{2n+2})).}
	\end{array}$$
	Thus, since $d_V(t)\leq d_V(t_{2n+2})$ from \eqref{dist-1cs}, we can deduce that
	$$\frac{d}{dt}\mathcal{L}(t)\leq 0,\quad \forall t\in (t_{2n+1},t_{2n+2}).$$
	As a consequence, for all $t_{2n+1}<s<t\leq t_{2n+2}$, it comes that
	$$\mathcal{L}(t)\leq \mathcal{L}(s).$$ 
	Letting $s\to t_{2n+1}^+$, we get
	$$\mathcal{L}(t)\leq \lim_{s\to t_{2n+1}^+}\mathcal{L}(s)=\lim_{s\to t_{2n+1}^+}D(s)+\int_{0}^{\underset{s\in [0,t_{2n+1}]}{\max}d_{X}(s)}\min\left\{e^{-\tilde KT}\min_{\sigma\in [0,r]}\tilde\psi(\sigma),\frac{e^{-\tilde KT}}{T}\right\}dr,$$
	for all $t\in (t_{2n+1},t_{2n+2}]$. Thus, using \eqref{limt2n+1}, we can write
	$$\begin{array}{l}
		\vspace{0.3cm}\displaystyle{\mathcal{L}(t)\leq \frac{e^{\tilde K(t_{2n+2}-t_{2n+1})}}{2-e^{\tilde K(t_{2n+2}-t_{2n+1})}}\mathcal{D}(t_{2n+1})+\int_{0}^{\underset{s\in [0,t_{2n+1}]}{\max}d_{X}(s)}\min\left\{e^{-\tilde KT}\min_{\sigma\in [0,r]}\tilde\psi(\sigma),\frac{e^{-\tilde KT}}{T}\right\}dr}\\	
		\displaystyle{\hspace{4.5cm}\leq \frac{e^{\tilde K(t_{2n+2}-t_{2n+1})}}{2-e^{\tilde K(t_{2n+2}-t_{2n+1})}}\mathcal{L}(t_{2n+1}),}
	\end{array}$$
	for all $t\in (t_{2n+1},t_{2n+2}]$. 
	Therefore, 
	\begin{equation}\label{lia2}
		\mathcal{L}(t)\leq \frac{e^{\tilde K(t_{2n+2}-t_{2n+1})}}{2-e^{\tilde K(t_{2n+2}-t_{2n+1})}}\mathcal{L}(t_{2n+1}),\quad \forall t\in [t_{2n+1},t_{2n+2}].
	\end{equation}
	Now, combining \eqref{lia1} and \eqref{lia2}, it turns out that
	\begin{equation}\label{lia3}
		\mathcal{L}(t_{2n+2})\leq \frac{e^{\tilde K(t_{2n+2}-t_{2n+1})}}{2-e^{\tilde K(t_{2n+2}-t_{2n+1})}}\frac{1}{\left(1-\int_{t_{2n-1}}^{t_{2n}}\phi(s)ds\right)}\mathcal{L}(t_{2n}),\quad \forall n\in \mathbb{N}.
	\end{equation}
	Thus, thanks to an induction argument, from \eqref{lia3} it follows that 
	\begin{equation}\label{lia4}
			\mathcal{L}(t_{2n+2})\leq \prod_{p=1}^{n}\left(\frac{e^{\tilde K(t_{2p+2}-t_{2p+1})}}{2-e^{\tilde K(t_{2p+2}-t_{2p+1})}}\frac{1}{\left(1-\int_{t_{2p-1}}^{t_{2p}}\phi(s)ds\right)}\right)\mathcal{L}(t_2),
	\end{equation}
	for all $n\in \mathbb{N}$. 
	\\Now, let $t\geq t_4$. Then, there exists $n\geq 2$ such that $t\in [t_{2n},t_{2n+2}]$. As a consequence, if $t\in [t_{2n},t_{2n+1}]$, from \eqref{lia1} and \eqref{lia4} with $n-1\geq 1$, we get
	$$\begin{array}{l}
		\vspace{0.3cm}\displaystyle{\mathcal{L}(t)\leq \frac{1}{1-\int_{t_{2n-1}}^{t_{2n}}\phi(s)ds}\mathcal{L}(t_{2n})\leq \frac{e^{\tilde{K}(t_{2n+2}-t_{2n+1})}}{2-e^{\tilde{K}(t_{2n+2}-t_{2n+1})}}\frac{1}{1-\int_{t_{2n-1}}^{t_{2n}}\phi(s)ds}\mathcal{L}(t_{2n})}\\
		\vspace{0.3cm}\displaystyle{\hspace{0.8cm}\leq \frac{e^{\tilde{K}(t_{2n+2}-t_{2n+1})}}{2-e^{\tilde{K}(t_{2n+2}-t_{2n+1})}}\frac{1}{1-\int_{t_{2n-1}}^{t_{2n}}\phi(s)ds}\prod_{p=1}^{n-1}\left(\frac{e^{\tilde K(t_{2p+2}-t_{2p+1})}}{2-e^{\tilde K(t_{2p+2}-t_{2p+1})}}\frac{1}{\left(1-\int_{t_{2p-1}}^{t_{2p}}\phi(s)ds\right)}\right)\mathcal{L}(t_2)}\\
		\displaystyle{\hspace{0.8cm}=\prod_{p=1}^{n}\left(\frac{e^{\tilde K(t_{2p+2}-t_{2p+1})}}{2-e^{\tilde K(t_{2p+2}-t_{2p+1})}}\frac{1}{\left(1-\int_{t_{2p-1}}^{t_{2p}}\phi(s)ds\right)}\right)\mathcal{L}(t_2).}
	\end{array}$$
On the other hand, if $t\in [t_{2n+1},t_{2n+2}]$, from \eqref{lia1}, \eqref{lia2} and \eqref{lia4} it comes that
$$\begin{array}{l}
	\vspace{0.3cm}\displaystyle{\mathcal{L}(t)\leq \frac{e^{\tilde K(t_{2n+2}-t{2n+1})}}{2-e^{\tilde K(t_{2n+2}-t{2n+1})}}\mathcal{L}(t_{2n+1})\leq \frac{e^{\tilde K(t_{2n+2}-t{2n+1})}}{2-e^{\tilde K(t_{2n+2}-t{2n+1})}}\frac{1}{1-\int_{t_{2n-1}}^{t_{2n}}\phi(s)ds}\mathcal{L}(t_{2n})}\\
	\displaystyle{\hspace{1.5cm}\leq \prod_{p=1}^{n}\left(\frac{e^{\tilde K(t_{2p+2}-t_{2p+1})}}{2-e^{\tilde K(t_{2p+2}-t_{2p+1})}}\frac{1}{\left(1-\int_{t_{2p-1}}^{t_{2p}}\phi(s)ds\right)}\right)\mathcal{L}(t_2).}
\end{array}$$
Thus, for all $t\geq t_4$,
\begin{equation}\label{lt}
	\begin{split}
		\mathcal{L}(t)\leq \prod_{p=1}^{n}\left(\frac{e^{\tilde K(t_{2p+2}-t_{2p+1})}}{2-e^{\tilde K(t_{2p+2}-t_{2p+1})}}\frac{1}{\left(1-\int_{t_{2p-1}}^{t_{2p}}\phi(s)ds\right)}\right)\mathcal{L}(t_2)\\=e^{\sum_{p=1}^{n}\left[\ln\left(\frac{e^{\tilde K(t_{2p+2}-t_{2p+1})}}{2-e^{\tilde K(t_{2p+2}-t_{2p+1})}}\right)+\ln\left(\frac{1}{1-\int_{t_{2p-1}}^{t_{2p}}\phi(s)ds}\right)\right]}\mathcal{L}(t_2).
	\end{split}
\end{equation}
Now, from \eqref{sum1tilde}, $\sum_{p=1}^{+\infty}\ln\left(\frac{e^{\tilde K(t_{2n+2}-t_{2p+1})}}{2-e^{\tilde{K}(t_{2p+2}-t_{2p+1})}}\right)<+\infty$. Also, $\sum_{p=1}^{+\infty}\ln\left(\frac{1}{1-\int_{t_{2p-1}}^{t_{2p}}\phi(s)ds}\right)<+\infty$. Indeed, from \eqref{boundbc}, it turns out that $$\int_{t_{2p-1}}^{t_{2p}}\phi(s)ds\leq \frac{e^{-\tilde K T}}{T}(t_{2p}-t_{2p-1}) ,\quad \forall p\geq 1.$$	
	Then, since \eqref{sum1tilde} holds (which implies \eqref{badtilde}), we have that $t_{2p}-t_{2p-1}\to 0$, as $p\to \infty$, and we can write
	$$
\begin{array}{l}
\displaystyle{\ln\left(\frac{1}{1-\int_{t_{2p-1}}^{t_{2p}}\phi(s)ds}\right)\leq \ln\left(\frac{1}{1-\frac{e^{-\tilde K T}}{T}(t_{2p}-t_{2p-1})}\right)}\\
\displaystyle{\hspace{1.5 cm}
=-\ln \left(1-\frac{e^{-\tilde K T}}{T}(t_{2p}-t_{2p-1})\right)
\sim \frac{e^{-\tilde K T}}{T}(t_{2p}-t_{2p-1}).}
\end{array}
$$
	As a consequence, since $\sum_{p=1}^{+\infty}(t_{2p}-t_{2p-1})<+\infty$ from \eqref{badtilde}, it holds $\sum_{p=0}^{+\infty}\ln\left(\frac{1}{1-\int_{t_{2p-1}}^{t_{2p}}\phi(s)ds}\right)<+\infty$.
	\\So, setting $$C:=e^{\sum_{p=0}^{\infty}\left[\ln\left(\frac{e^{\tilde K(t_{2n+2}-t_{2p+1})}}{2-e^{\tilde K(t_{2p+2}-t_{2p+1})}}\right)+\ln\left(\frac{1}{1-\int_{t_{2p-1}}^{t_{2p}}\phi(s)ds}\right)\right]}\mathcal{L}(t_2),$$
	taking into account of \eqref{lt}, we can conclude that
	\begin{equation}\label{lia5}
		\mathcal{L}(t)\leq C,\quad \forall t\geq t_4.
	\end{equation}
	As a consequence, for all $t\geq t_4$, by definition of $\mathcal{L}$,
	$$\int_{0}^{\underset{s\in [0,t]}{\max}d_{X}(s)}\min\left\{e^{-\tilde KT}\min_{\sigma\in [0,r]}\tilde\psi(\sigma),\frac{e^{-\tilde KT}}{T}\right\}dr\leq \mathcal{L}(t)\leq C.$$ 
	Letting $t\to \infty$ in the above inequality, we finally get \begin{equation}\label{lim2}
		\int_{0}^{\underset{s\in [0,\infty)}{\sup}d_{X}(s)}\min\left\{e^{-\tilde KT}\min_{\sigma\in [0,r]}\tilde\psi(\sigma),\frac{e^{-\tilde KT}}{T}\right\}dr\leq C.
	\end{equation} 
	Then, since the function $\tilde \psi$ satisfies  \eqref{infint}, from \eqref{lim2}, there exists a positive constant $d^{*}$ such that \begin{equation}\label{firstcond}
		\underset{s\in [0,\infty)}{\sup}d_{X}(s)\leq d^{*}.
	\end{equation}
	Now, we define
	$$\phi^{*}:=\min\left\{e^{-\tilde KT}\psi_{*},\frac{e^{-\tilde KT}}{T}\right\},$$
	where $$\psi_{*}=\min_{r\in[0,d^{*}]}\tilde\psi(r).$$
	Note that $\phi^{*}>0$, being $\tilde\psi$ a positive function. Also, from \eqref{firstcond}, it comes that $$\psi_{*}\leq \min\left\{\tilde\psi(r):r\in \left[0,\max_{s\in [0,t]}d_{X}(s)\right]\right\}=\tilde\psi_{t},$$
	for all $t\geq 0$. Thus, we get $$\phi^{*}\leq \phi(t), \quad \forall t\geq 0.$$
	As a consequence, for all $n\in \mathbb{N}_0$, $$\int_{t_{2n}}^{t_{2n+1}}\phi(s)ds\geq \phi^*(t_{2n+1}-t_{2n}),$$
	from which 
	$$1-\int_{t_{2n}}^{t_{2n+1}}\phi(s)ds\leq 1-\phi^*(t_{2n+1}-t_{2n}).$$
	So, recalling of inequality \eqref{n-2cs}, we can write
	\begin{equation}\label{n-2phi*}
		d_V(t_{2n+1})\leq \left(1-\phi^*(t_{2n+1}-t_{2n})\right)d_V(t_{2n}),\quad\forall n \in \mathbb{N}_0.
	\end{equation}
	Now, using \eqref{boundgrowthcs} and \eqref{n-2phi*}, we have that
	$$d_V(t_{2n+2})\leq \frac{e^{\tilde K(t_{2n+2}-t_{2n+1})}}{2-e^{\tilde K(t_{2n+2}-t_{2n+1})}}d_V(t_{2n+1})\leq \frac{e^{\tilde K(t_{2n+2}-t_{2n+1})}}{2-e^{\tilde K(t_{2n+2}-t_{2n+1})}}\left(1-\int_{t_{2n}}^{t_{2n+1}}\phi(s)ds\right)d_V(t_{2n})$$
	$$\leq \frac{e^{\tilde K(t_{2n+2}-t_{2n+1})}}{2-e^{\tilde K(t_{2n+2}-t_{2n+1})}}\left(1-\phi^*(t_{2n+1}-t_{2n})\right)d_V(t_{2n}).$$
	Thus, using an induction argument, we get
	$$d_V(t_{2n+2})\leq \prod_{p=0}^{n}\left(\frac{e^{\tilde K(t_{2p+2}-t_{2p+1})}}{2-e^{\tilde K(t_{2p+2}-t_{2p+1})}}\left(1-\phi^*(t_{2p+1}-t_{2p})\right)\right)d_V(0)$$
	$$=e^{\sum_{p=0}^{\infty}\ln \left(\frac{e^{\tilde K(t_{2p+2}-t_{2p+1})}}{2-e^{\tilde K(t_{2p+2}-t_{2p+1})}}\left(1-\phi^*(t_{2p+1}-t_{2p})\right)\right)}d_V(0)$$
	$$=e^{\sum_{p=0}^{\infty}\left[\ln \left(\frac{e^{\tilde K(t_{2p+2}-t_{2p+1})}}{2-e^{\tilde K(t_{2p+2}-t_{2p+1})}}\right)+\ln\left(1-\phi^*(t_{2p+1}-t_{2p})\right)\right]}d_V(0)$$
	Now, $\sum_{p=0}^{\infty}\ln \left(\frac{e^{\tilde K(t_{2p+2}-t_{2p+1})}}{2-e^{\tilde K(t_{2p+2}-t_{2p+1})}}\right)<+\infty$ from \eqref{sum1tilde}. Then, the solution $\{x_{i},v_{i}\}_{i=1,\dots,N}$ exhibits asymptotic flocking if the following condition is satisfied:
	\begin{equation}\label{psi*}
		\sum_{p=0}^{\infty}\ln\left(\left(1-\phi^*(t_{2p+1}-t_{2p})\right)\right)=-\infty.
	\end{equation}
	However, the above condition is guaranteed since, from \eqref{tnbuoni},
	$$1-\phi^*(t_{2p+1}-t_{2p})\leq 1-\frac{\phi^*}{\tilde K}\in (0,1),\quad \forall p\in \mathbb{N}_0.$$
	Thus, $$\sum_{p=0}^{\infty}\ln\left(\left(1-\phi^*(t_{2p+1}-t_{2p})\right)\right)\leq \sum_{p=0}^{\infty}\ln\left(1-\frac{\phi^*}{\tilde K}\right)=-\infty,$$
	from which \eqref{psi*} is fulfilled.
\end{proof}

\bigskip
\noindent {\bf Acknowledgements.} We thank INdAM-GNAMPA and UNIVAQ  for the support.

\end{document}